\documentclass[12pt]{amsart}
\usepackage[toc,page]{appendix}
\usepackage{amsmath,amssymb,latexsym}
\usepackage[small]{caption}
\usepackage{graphicx,color,mathrsfs,tikz}
\usepackage{subfigure,color}
\usepackage{cite}
\usepackage[colorlinks=true,urlcolor=blue,
citecolor=red,linkcolor=blue,linktocpage,pdfpagelabels,
bookmarksnumbered,bookmarksopen]{hyperref}

\usepackage{soul}
\usepackage[T1]{fontenc}
\usepackage[polish,russian,italian,english]{babel}
\usepackage[utf8]{inputenc}

\usepackage{enumitem}
\usepackage[left=2.65cm,right=2.65cm,top=2.9cm,bottom=2.9cm]{geometry}
\usepackage[hyperpageref]{backref}

\usepackage[colorinlistoftodos,prependcaption]{todonotes}

\numberwithin{equation}{section}
\newtheorem{theorem}{Theorem}[section]

\newtheorem{lemma}[theorem]{Lemma}

\theoremstyle{definition}

\renewcommand{\epsilon}{\eps}
\renewcommand{\i}{{\rm i}}

\newcommand{\B}{{\mathcal B}}

\renewcommand{\H}{{\mathcal H}}

\newcommand{\Co}{{\mathcal C}}
\newcommand{\N}{{\mathbb N}}
\newcommand{\R}{{\mathbb R}}

\newcommand{\eps}{\varepsilon}

\newcommand{\pnorm}[2][]{\if #1'' \left|#2\right|_p \else \left|#2\right|_{#1} \fi}

\renewcommand{\theta}{\vartheta}

\newcommand{\Rn}{{\mathbb R^{n}}}
\newcommand{\eqlab}[1]{\begin{equation}  \begin{aligned}#1 \end{aligned}\end{equation}} 
\newcommand{\bgs}[1]{\begin{equation*} \begin{aligned}#1\end{aligned}\end{equation*}} 
 \newcommand{\syslab}[2] []  {\begin{equation}#1  \left\{\begin{aligned}#2\end{aligned}\right.\end{equation}} 
  \newcommand{\sys}[2][]{\begin{equation*}#1  \left\{\begin{aligned}#2\end{aligned}\right.\end{equation*}}
  
\def\Xint#1{\mathchoice
{\XXint\displaystyle\textstyle{#1}}%
{\XXint\textstyle\scriptstyle{#1}}%
{\XXint\scriptstyle\scriptscriptstyle{#1}}%
{\XXint\scriptscriptstyle\scriptscriptstyle{#1}}%
\!\int}
\def\XXint#1#2#3{{\setbox0=\hbox{$#1{#2#3}{\int}$ }
\vcenter{\hbox{$#2#3$ }}\kern-.6\wd0}}

\def\dashint{\Xint-}

\DeclareMathOperator{\dist}{dist}
\def\Xint#1{\mathchoice
{\XXint\displaystyle\textstyle{#1}}%
{\XXint\textstyle\scriptstyle{#1}}%
{\XXint\scriptstyle\scriptscriptstyle{#1}}%
{\XXint\scriptscriptstyle\scriptscriptstyle{#1}}%
\!\int}
\def\XXint#1#2#3{{\setbox0=\hbox{$#1{#2#3}{\int}$ }
\vcenter{\hbox{$#2#3$ }}\kern-.6\wd0}}

\def\dashint{\Xint-}

\title[On the mean value property of fractional harmonic functions]{On the mean value property \\ of fractional harmonic functions}\thanks{
{\em Claudia Bucur}: Istituto Nazionale di Alta Matematica,
Piazzale Aldo Moro 5,
00185 Rome, Italy, and
Dipartimento di Matematica, Universit\`a di Milano,
Via Saldini 50, 20133 Milan, Italy. {\tt claudia.bucur@aol.com}
\\
{\em Serena Dipierro}:
Department of Mathematics
and Statistics,
University of Western Australia,
35 Stirling Hwy, Crawley WA 6009, Australia.
{\tt serena.dipierro@uwa.edu.au}\\
{\em Enrico Valdinoci}:
Department of Mathematics
and Statistics,
University of Western Australia,
35 Stirling Hwy, Crawley WA 6009, Australia. {\tt enrico.valdinoci@uwa.edu.au}\\
It is a pleasure to 
thank Krzysztof Bogdan for extremely
pleasant and very beneficial scientific discussions.
We are also indebted to the Referees for their
useful comments who helped us improve
this paper.
The authors are members of INdAM.
The first author is supported by the INdAM Starting Grant
``PDEs, free boundaries, nonlocal equations and applications''.
The second and third authors
are members of AustMS and
are supported by the Australian Research Council
Discovery Project DP170104880 NEW ``Nonlocal Equations at Work''.
The second author is supported by
the Australian Research Council DECRA DE180100957
``PDEs, free boundaries and applications''. Part of this work was carried
out during a very pleasant and fruitful visit of the first author to the
University of Western Australia, which we thank for the warm hospitality.}

\author{Claudia Bucur}
\author{Serena Dipierro}
\author{Enrico Valdinoci}


\begin{document}

\begin{abstract}
As well known, harmonic functions satisfy
the mean value property,
i.e. the average of such a function over a ball
is equal to its value at the center. 
This fact naturally raises the question on whether this is a feature characterizing only  balls,
namely, is a set, for which all harmonic functions
satisfy the mean value property, necessarily a ball?

This question was investigated by several authors,
including Bernard Epstein [Proc. Amer. Math. Soc., 1962],
Bernard Epstein and Menahem Max
Schiffer [J. Anal. Math., 1965],
Myron Goldstein and Wellington H. Ow
[Proc. Amer. Math. Soc., 1971], who 
obtained a positive answer to this question under suitable additional
assumptions.

The problem was finally elegantly, completely and positively
settled
by \"Ulk\"u
Kuran [Bull. London Math. Soc., 1972], with an artful use
of elementary techniques.

This classical problem has been recently fleshed out
by Giovanni Cupini, Nicola Fusco,
Ermanno Lanconelli and Xiao Zhong [J. Anal. Math., in press]
who proved a quantitative stability result for the mean value formula,
showing that a suitable
``mean value gap'' (measuring the normalized
difference between the average of harmonic functions on a given set
and their pointwise value) is bounded from below by the Lebesgue measure of the ``gap'' between the set and the ball (and, consequently, by
the Fraenkel asymmetry of the set). That is,
if a domain ``almost'' satisfies the mean value property for all harmonic functions, then
that domain is ``almost'' a ball.

The goal of this note is to investigate some nonlocal counterparts of these results. 
Some of our arguments rely on fractional potential theory,
others
on purely nonlocal properties, with no classical
counterpart, such as the fact that
``all functions are locally
fractional harmonic up to a small error''.
\end{abstract}
\maketitle

\tableofcontents

\section{Introduction}

\subsection{A fractional version of inverse
mean value properties}
A classical question, dating back to the works~\cite{MR0140700, MR0177124, MR0279320},
is to determine under which conditions a domain, 
providing a mean value property 
for every harmonic function, needs to be necessarily a ball.

More precisely, it is well known that if~$u$ is harmonic
in a domain, then $u$ satisfies the mean value property on every ball compactly contained in that domain. Precisely, say the closure
of a ball~$B_r$ (centered at the origin) is contained in the domain, then
\begin{equation*} \label{0-A}
u(0)=\dashint_{B_r}u(y)\,dy,\end{equation*}
where, as usual, the ``dashed'' integral symbol stands for the average.

The mean value property 
is certainly remarkable and of great importance in 
 the classical theory of harmonic functions. 
A natural question is 
to consider an ``inverse problem'' and try to classify
all domains for which a mean value formula can hold: namely,
if~$\Omega$ is a given domain of~$\R^n$ containing the origin
and with the property that
\begin{equation} \label{0-B} u(0)=\dashint_{\Omega}u(y)\,dy\end{equation}
for all functions~$u$ that are harmonic in~$\Omega$,
is it possible to say anything about~$\Omega$?
That is, how ``special'' are
the domains satisfying~\eqref{0-B}?

This problem was definitely settled by~\"{U}lk\"{u}
Kuran in~\cite{MR0320348}, who established, with a concise and very elegant proof,
that if~$\Omega$ is a bounded domain, containing the origin, such that \eqref{0-B} holds  for every harmonic, integrable function~$u$ in~$\Omega$,
then~$\Omega$ is a ball centered at the origin. 

As a matter of fact, the work in~\cite{MR0320348}
was the climax of a rather intense research in the sixties and seventies,
that started with~\cite{MR0140700}, in which
the classification result for domains satisfying~\eqref{0-B}
was obtained under the additional assumption
that~$\Omega$ was simply connected.
The simple connectivity assumption
was later replaced in~\cite{MR0177124} 
by the hypothesis that the complement of~$\Omega$
possesses a nonempty interior.
Also, in~\cite{MR0279320} the classification result
was obtained
for planar domains with
at least one boundary component which is a continuum.
Interestingly, not only the result in~\cite{MR0320348}
completed the previous works in~\cite{MR0140700, MR0177124, MR0279320},
but it also presented an elementary\footnote{\label{gui1}For completeness, let\label{HASSAms}
us briefly recall the proof in~\cite{MR0320348}:
up to a dilation, we can suppose that~$B_1\subset\Omega$,
with~$\tilde x\in(\partial \Omega)\cap(\partial B_1)$. Then, let
$$ h(x):=\frac{|x|^2-1}{|x-\tilde x|^n}+1.$$
Since~$h(0)=0$, $h\ge1$ in~$\R^n\setminus B_1$,
and~$h$ is harmonic in~$\R^n\setminus\{\tilde x\}$, using~\eqref{0-B} twice
(once for~$\Omega$ and once for~$B_1$),
it follows that
\begin{eqnarray*} &&0=|\Omega|\,h(0)=\int_{\Omega}h(y)\,dy
=\int_{B_1}h(y)\,dy+\int_{\Omega\setminus B_1}h(y)\,dy\\&&\qquad\qquad=
|B_1|\,h(0)+\int_{\Omega\setminus B_1}h(y)\,dy
=
\int_{\Omega\setminus B_1}h(y)\,dy\ge |\Omega\setminus B_1|,\end{eqnarray*}
therefore~$|\Omega\setminus B_1|=0$ and thus~$\Omega=B_1$.}
approach to the question based on the Poisson Kernel of the ball.
\smallskip

Besides its theoretical interest, the result in~\cite{MR0320348}
has also natural consequences in game theory,
since the expected payoff of a random walk with prizes placed
at the boundary of a domain is clearly related to harmonic functions, and thus the mean value property of harmonic functions in this context translates into the possibility of exchanging the average expected payoff in a given region with the pointwise expected payoff calculated at a special point of that region,
concretely, the center of the ball (and Kuran's result states
that this reduction is not possible either with other regions,
or with other points of the ball).\smallskip 

We also remark that, denoting by~${\mathcal{H}}^{n-1}$
the $(n-1)$-dimensional
Hausdorff measure,
an interesting variant of Kuran's result is as follows:
if~$\Omega$ is a given domain of~$\R^n$,
containing the origin
and with the property that
\begin{equation} \label{0-B-B} u(0)=\dashint_{\partial\Omega}u(y)\,d{\mathcal{H}}^{n-1}(y)\end{equation}
for all functions~$u$ that are harmonic in~$\Omega$,
then~$\Omega$ is necessarily a ball:
this was proved in Theorem~III.2 of~\cite{MR1021402}
(see also~\cite{MR2441608, MR2468442, MR3977217}). 
Interestingly, the condition in~\eqref{0-B-B} can also be
classically dealt with ``dual'' formulations involving
a prescription on the normal derivative of the Green function of~$\Omega$
(see Theorem III.1 in~\cite{MR1021402},
Section~7 in~\cite{MR3893584}, and the references therein).
Related results are contained in~\cite{MR2746441, MR2747460}.
See also~\cite{MR1321628}
for a classical survey on the spherical and volume averages
of harmonic functions.
\smallskip

\medskip

In this paper we begin to investigate 
some possible fractional counterparts of these classical results. 
Precisely, we plan to answer this question:  if the value of any fractional harmonic function
at a given point equals a suitable fractional mean value on a domain, is that domain the ball
centered at the given point?
We will then focus on some quantitative versions of this question
in terms of different possible ``gap''
functions.

These results are somehow reminiscent of the classical
mean value formula and of Kuran-type problems,
and we address two types of nonlocal results.
A first proposal regards the well-known
fractional mean value formula on the ball.
Our result in Theorem~\ref{THM}
relies on this special measure,
which is endowed with extra information linked
to the spherical behavior. In this sense, the analysis
 appears somewhat more
specialized than in the classical case, where the Lebesgue measure is used instead.

In a second result given in Theorem \ref{2DERe},
we consider a different problem,
in which the measure in the mean value formula
is instead modeled on 
the Poisson kernel (that is, the density function of the fractional
harmonic
measure). In this setting,
 we provide partial results
related
 to whether the limit behavior of the Poisson kernel at boundary points is a constant.\medskip

Some of our proofs deeply rely
on the potential theory of fractional operators,
as developed in~\cite{MR126885,
MR1438304, chengrennest, MR1654115, MR1980119, MR2006232, bogdan1, MR2256481, MR2365478, MR2569321, MR4061422}.
Other proofs take
advantage of
the particular structure of nonlocal equations, in particular we 
exploit the main result of~\cite{DSV14}: any smooth function locally approximates a fractional harmonic function. In this way,  we construct a fractional harmonic function with the desired properties (that plays the role that the Poisson kernel played in the proof of Kuran, see footnote at page \pageref{gui1}). We think that this is a nice example of how, in some occasions, 
the nonlocal setting 
provides a technical and conceptual simplification with respect to the classical case.
\medskip

In terms of motivation
and application of fractional mean value
formulas, we also mention that a fractional version of the expected
payoff game with {L}\'{e}vy processes
in a given domain and prizes set in the complement of the domain are described
in detail, for instance, in Chapter~2.2
of~\cite{nonlocal}.

\medskip

To state our fractional versions of inverse mean value properties,
we introduce some notations and preliminary notions. 
Here and in the rest of the paper $\Omega \subset \Rn$ is a bounded open set and $s\in(0,1)$ is a fixed number. 
Moreover, as customary,
we use the notation~$\Co \Omega:=\R^n\setminus\Omega$.

We recall that a function~$u:\R^n\to\R$
(say, for simplicity,
sufficiently smooth in a given
domain $\Omega\subset \Rn$),
satisfying
\begin{equation*}\label{L1s}
\int_{\R^n} \frac{|u(y)|}{1+|y|^{n+2s}}\,dy<+\infty
\end{equation*}
is $s$-harmonic in~$\Omega$  if 
\[ (-\Delta)^s u = 0 \qquad \mbox{ in }\; \; \Omega,\]
where
\[ (-\Delta)^s u(x) :=P.V.\int_{\Rn} \frac{u(x)-u(x-y)}{|y|^{n+2s}}\, dy \]
is the fractional Laplace operator (see, for instance,~\cite{nonlocal,mateusz,gettinacq}). 
 	It is known, as in the classical case, that a function~$u$ is $s$-harmonic in $\Omega$ if and only if $u$
possesses the following mean value property:
\eqlab{ \label{mvpx} u(0) = 
c(n,s)\,\int_{\Co B_r} \frac{ r^{2s}\, u(y)}{(|y|^2-r^2)^s|y|^n}\, dy ,}
 for any $r>0$ such that ${B_r} \subset\subset \Omega$ 
 (see~\cite[Theorem 2.1]{Abatangelo},~\cite[Lemma A.6]{BucurGreen}
or~\cite[Chapter 1.6]{Landkof}). Here, the notation~$c(n,s)$
stands for a positive, normalizing constant.
In particular, taking~$u:=1$ in~\eqref{mvpx} and
setting
\begin{equation} \label{musta1}
d\mu_{r}(y):=
\frac{c(n,s)\,
	r^{2s}\, dy}{(|y|^2 -r^2)^s |y|^{n} },
\end{equation}
then~$\mu_r$ is a measure on~$\Co B_r$, with
\begin{equation} \label{998y99iuqhjss89}
\mu_r(\Co B_r)=1.
\end{equation}
In this framework, we can write~\eqref{mvpx} in the form \begin{equation}
\label{mvp} u(0) = 
\int_{\Co B_r} u(y)\, d\mu_r(y)=
\frac{1}{\mu_r(\Co B_r)}\int_{\Co B_r} u(y)\, d\mu_r(y),\end{equation}
for any $r>0$ such that 
\begin{equation}\label{LemmaJAK:AA3440}
{B_r} \subset\subset \Omega.\end{equation}
As a matter of fact, if in addition~$u\in C(\R^n)$,  
then~\eqref{LemmaJAK:AA3440} can be replaced
by the weaker condition that
\begin{equation}\label{LemmaJAK:AA34400}
{B_r} \subset\Omega,\end{equation}
see Lemma~\ref{LemmaJAK:AA344}
in Appendix~\ref{UNAryty45666}.
\medskip

We now discuss a suitable inverse problem for~\eqref{mvp}.
To state it, we define
\begin{equation}\label{SPAZ} 
\mathcal H^s(\Omega):= \left\{ u \in 
C(\R^n) \mbox{ s.t.~}
\; \int_{\R^n} \frac{|u(y)|}{1+|y|^{n+2s}}\,dy<+\infty
\, \mbox{ and }\; (-\Delta )^s u =0  \mbox{ in } \Omega\right\}.
\end{equation}
In this setting, we have the following result.

\begin{theorem}\label{THM}
Let $\Omega \subset \Rn$ be a bounded open set, containing the origin,
and define
\begin{equation}\label{ERRE}
r:={\rm{dist}}(0, \partial \Omega).
\end{equation}
Suppose that
\begin{equation}\label{OMEGA}
u(0)=\frac{1}{\mu_r(\Co \Omega)}\int_{\Co \Omega} u(y) \, d\mu_r (y)
\end{equation}
for all functions~$u\in\mathcal H^s(\Omega)$.

Then
\begin{equation*}\label{BALL}
 \Omega = B_r .
 \end{equation*} 
\end{theorem}

In short, Theorem~\ref{THM} says that if~$\Omega$
satisfies a fractional mean value property (compare~\eqref{mvp}
and~\eqref{OMEGA}) with respect to a suitable measure, then~$\Omega$
is necessarily a ball. Two proofs of
Theorem~\ref{THM} are provided
in Section \ref{uno}, using both a typically nonlocal and a
potential theoretic approach.
\medskip

We point out that if~$r$ is as in~\eqref{ERRE},
then~\eqref{LemmaJAK:AA34400} is satisfied
(but~\eqref{LemmaJAK:AA3440} does not hold,
and this makes the result in Lemma~\ref{LemmaJAK:AA344}
technically important for our goals).
\medskip

On the one hand, we can consider the setting in~\eqref{OMEGA}
as a nonlocal transposition of that in~\eqref{0-B-B},
in which the classical averages along the boundary of the domain
(corresponding to classical Dirichlet conditions)
are replaced by suitable fractional
averages in the exterior of the domain (corresponding to fractional Dirichlet conditions,
which are indeed external, and not boundary, prescriptions).
 On the other hand, we stress that the special role played by the fractional
harmonic mean formulas here is quite different than in the classical
case, in which the surface measure on~$\partial\Omega$
is not the
restriction of the surface measure on~$\partial B_r$ to~$\partial\Omega$.
\medskip

We remark that the situation in Theorem~\ref{THM}
would be completely different if one replaced~\eqref{OMEGA}
with a similar formula holding for a suitable measure~$\mu$,
of the type
\begin{equation*}\label{SW} u(0)=\int_{\Co \Omega} u(y) \, d\mu (y).\end{equation*}
Indeed, 
this problems is structurally very
different from the setting in~\eqref{OMEGA},
since 
it is related to the ``balayage'' problems
for fractional harmonic functions and hold true
by taking~$\mu$ as the fractional harmonic measure
(see e.g.~\cite[formula~(4.5.9) and Theorem~4.16]{Landkof},
\cite[Lemma~17]{MR1438304},
\cite[Theorem~7.2]{MR3888401},
\cite[Section~2.2]{MR2365802},
\cite[Remark~3.1]{MR3750233},
and also~\cite{MR1893056} and the references therein).
The fractional harmonic measure~$\mu$ 
is related to the Poisson kernel and
has also a probabilistic interpretation, being
the distribution of a L\'evy process started at the origin
and stopped when exiting the domain~$\Omega$.
In general, these considerations highlight the importance
of carefully choosing the measure~$\mu_r$ in~\eqref{OMEGA}
if one is interested in classification results for the domain~$\Omega$,
which would not be valid for other types of measures (e.g.,
for the fractional harmonic measure).
\medskip

Other lines of investigation related to harmonic measures 
and more generally to averages
of subharmonic and superharmonic functions with respect
to different measures, are linked to the notion
of Jensen measures, see e.g.~\cite{MR1873590, MR1876284, MR2379689}. 
For other type of classification results
concerning different averages, see~\cite{MR2584983},
and also~\cite{MR637445, MR860918, MR1021402}.
\medskip

The proof of Theorem~\ref{THM}
is contained in Section~\ref{thm11}.

\subsection{Stability results for the fractional mean value property}

Another natural development is to establish 
a quantitative version of Theorem~\ref{THM}
in view of the stability results in~\cite[Theorem 1.1]{Cup}
for the classical case.  We 
plan to understand whether the fact that
every $s$-harmonic function is ``close'' to
its mean value on a domain implies that 
the domain is necessarily ``close'' to being a ball,
or vice-versa, if the ``distance'' between
the pointwise value of an $s$-harmonic functions and
its mean value on a domain
remains bounded away from zero, unless the domain is a ball.
\medskip

  To this end, 
we introduce several notions of ``gaps'',
which in turn will provide structurally different results.
We consider $\Omega \subset \Rn$ to be a bounded and open set
containing the origin, and we denote $r$ as in~\eqref{ERRE}
and~$\mu_r$ as in~\eqref{musta1}.
Taking inspiration from in~\cite[formula~(1.2)]{Cup},
we define the rescaled fractional Gauss mean value gap
\begin{equation}\label{F86243gerre} G_r(\Omega):=\sup_{u\in \H^s(\Omega)} \frac{\displaystyle \bigg|u(0)- \frac{1}{\mu_r(\Co \Omega)} \int_{\Co \Omega} u(y)\, d\mu_r (y)\bigg|}{\displaystyle \int_{\Co B_r} |u(y)| \, d\mu_r (y)}.\end{equation}
In light of~\eqref{mvp}, we know that balls
make the fractional Gauss mean value gap vanish. 
We prove that, conversely, all other
sets produce significant gaps,
and it is impossible to make~$G_r(\Omega)$ smaller
than a universal threshold, unless~$\Omega$ is a ball.
We state the precise quantitative result in the next theorem.

\begin{theorem} \label{STA5}
It holds that~$G_r(B_r)=0$, and that~$  G_r(\Omega)\ge1$ for all~$\Omega$ 
such that~$0\in \Omega$,
${\rm{dist}}(0, \partial \Omega)=r$
and~$\Omega\setminus B_r\ne \varnothing$. 
\end{theorem}

As a variation of the gap in~\eqref{F86243gerre},
one can consider
the quantity
\begin{equation*} \begin{split}
G^*_r(\Omega)\,&:=\sup_{u\in \H^s(\Omega)} \frac{\displaystyle \bigg|u(0)- \frac{1}{\mu_r(\Co \Omega)} \int_{\Co \Omega} u(y)\, d\mu_r (y)\bigg|}{\left|
\displaystyle \int_{\Co B_r} u(y) \, d\mu_r (y)\right|}
.\end{split}\end{equation*}
We observe that~$G_r(\Omega)\le G^*_r(\Omega)$, hence~$
G^*_r(\Omega)\ge1$ whenever $\Omega\setminus B_r\ne \varnothing$, in light of Theorem~\ref{STA5}.
We can sharpen this estimate and
prove that $G^*_r(\Omega)$ can become arbitrarily large.

\begin{theorem}\label{7u8j9234905iitt}
Let~$r>0$ and let~${\mathcal{U}}_r$ be the family of
bounded
open sets~$\Omega$, with~$C^\infty$ boundary such that~$0\in \Omega$ and ${\rm{dist}}(0, \partial \Omega)=r$.
Then~$G_r^*(B_r)=0$ and 
\begin{equation}\label{5t6t6g6h89g304}
\sup_{\Omega\in{\mathcal{U}}_r}G^*_r(\Omega)=+\infty.
\end{equation}
\end{theorem}

As a third fractional gap, we consider
\begin{equation}\label{Evisd923hhdgur945}
{\mathcal{G}}_r(\Omega):=
\sup_{{u\in \H^s(\Omega)}\atop{\|u\|_{L^\infty(\Omega)}\leq1}}
\bigg|\mu_r(\Co \Omega)\,u(0)- \int_{\Co \Omega} u(y)\, d\mu_r (y)\bigg|.\end{equation}
Differently from the previous gaps, ${\mathcal{G}}_r(\Omega)$
comprises a precise information with respect to~$\mu_r(\Omega\setminus B_r)$,
namely small values of~${\mathcal{G}}_r(\Omega)$
correspond to 
small values of $\mu_r(\Omega\setminus B_r)$, and viceversa
a small value of $\mu_r(\Omega\setminus B_r)$  produces a small ${\mathcal{G}}_r(\Omega)$. The precise result is the following.

\begin{theorem}\label{GSTOTRA}
There exists a universal number~$C>0$,
such that
$$ {\mathcal{G}}_r(\Omega)\in\left[ \frac{\mu_r( \Omega\setminus B_r)}{C},\,
C\mu_r( \Omega\setminus B_r)\right].$$
\end{theorem}

The proofs of these different stability theorems are the content of Section \ref{due}.

\subsection{A fractional inverse
mean value property involving a Poisson-like kernel}
We now investigate a different approach to the inverse
of fractional mean value properties, 
by considering a family of  
kernels
modeled on the fractional Poisson kernel. 
This family of  
kernels is parameterized by the set~$\Omega$; the kernels take into account the distance with respect to the boundary of~$\Omega$,
in a way that mimics the weight provided in the Poisson kernel
(depending on the distance to
the boundary of the domain). The reader can find
in Appendix B some notes on the Poisson kernel.

In the forthcoming results, we consider
only sets~$\Omega$ with~$C^{1,1}$ boundary, to fall within the cases studied in~\cite{MR1687746}.
We will also use the following notation for boundary limit: given a function~$f$ and a point $p\in \partial \Omega$, we write that
	\eqlab{\label{boundlim}
		 \lim_{{q\in{\mathcal{C}}\Omega}\atop{q\to p}}f(q)=\ell	
	}
if all the sequences accessing the boundary point in the exterior normal direction
approach the value~$\ell$, that is 
$$ \lim_{t\to0^+}f(p+t\nu(p))=\ell,$$
being~$\nu(p)$ the external normal of~$\Omega$ at~$p$
(other notions of non-normal limits
may be considered as well, see e.g. 
Theorem~3.2 in~\cite{MR1980119}).
\smallskip

We consider a function $\mathfrak F\colon (0,\infty)\times(0,\infty)\to \R_+$, and let $F_{\Omega} \colon \Co \overline \Omega \to \R_+$ be such that 
	\begin{equation}\label{FOME}
		F_\Omega(y): = \mathfrak c(\Omega)\; \mathfrak F\left( |y|, \dist(y,\partial \Omega)\right),
	\end{equation}
	where
	\begin{equation}\label{54327yYAS:03eorfijhh6778S}
		 \mathfrak c(\Omega) := \int_{\Co \Omega} F_\Omega(y) \, dy.
	\end{equation}
As a concrete example of the setting in~\eqref{FOME}, one could take 
	\eqlab{ \label{hop3}
		F_\Omega(y):=
\frac{c(n,s)}{|y|^n \,\big(
{\rm{dist}}(y, \partial \Omega)\big)^s\,\big(2+{\rm{dist}}(y, \partial \Omega)\big)^s},
	}
	with~$c(n,s)$ as in~\eqref{musta1}
and~\eqref{998y99iuqhjss89}. We observe that
the choice in~\eqref{hop3}
is consistent with the Poisson kernel on the ball:
namely,  if~$y\in{\mathcal{C}}B_1$, then~${\rm{dist}}(y, \partial B_1)=|y|-1$,
whence,
in~${\mathcal{C}}B_1$,
\begin{equation}\label{Nodnsckq2hwyrfy8w}
	F_{B_1}(y) =
		\frac{c(n,s)}{|y|^n \,\big(
		{\rm{dist}}(y, \partial B_1)\big)^s\,\big(2+{\rm{dist}}(y, \partial B_1)\big)^s}=
		\frac{c(n,s)}{|y|^n \,(
		|y|^2-1)^s}
		= P_{B_1 }(0,y) 
		,\end{equation}
thanks to \eqref{ballpoi}.\medskip

 In the framework described
in~\eqref{FOME}  we have the following result.

\begin{theorem}\label{DERe99}
Let $\Omega \subset \Rn$ be an open set with~$C^{1,1}$ boundary,
containing the origin, such that
	\begin{equation*} 
			{\rm{dist}}(0, \partial \Omega)=1,
	\end{equation*} 
	and let $p\in \partial \Omega \cap \partial B_1$.
Let $F_{\Omega} \colon \Co \overline \Omega \to \R_+$ be as in~\eqref{FOME} and assume that 
	\eqlab{\label{fru3}
		\lim_{{q\in{\mathcal{C}}\Omega}\atop{q\to p}} \frac{1}{\mathfrak c(\Omega)} F_\Omega(q) |q|^n \left(\dist(q,\partial \Omega)\right)^s= \frac{c(n,s)}{2^s}.
	}
Suppose that
	\begin{equation}\label{OMEGA-OME299}
		u(0)=\frac{1}{\mathfrak c(\Omega)} \int_{\Co \Omega} u(y) F_\Omega  (y) \, dy
	\end{equation}
for all functions~$u\in\mathcal H^s(\Omega)$. 
	Then
	\[ \Omega=B_1.\]
\end{theorem}

We make a few remarks on Theorem \ref{DERe99}. It is well known (see \eqref{DEF:POISSONKE}) that every $s$-harmonic function in $\Omega$ is uniquely determined by $P_\Omega$, the Poisson kernel in that domain. Basically (and we make this rigorous in the proof of Theorem~\ref{DERe99}), \eqref{OMEGA-OME299} is equivalent to asking for all $q\in \Co \overline \Omega$ that
	\eqlab{ \label{hop1}
		P_\Omega(0,q) \;\mathfrak c(\Omega)= F_\Omega(q).
	}
	Then \eqref{fru3} translates into 
	\eqlab{ \label{hop4}
			\lim_{{q\in{\mathcal{C}}\Omega}\atop{q\to p}} P_\Omega(0,q) |q|^n \left(\dist(q,\partial \Omega)\right)^s=\frac{c(n,s)}{2^s}.
	}
	 Furthermore, recalling the explicit formula for the Poisson kernel
on the ball in \eqref{ballpoi}, we have that
$$ \lim_{{q\in{\mathcal{C}}\Omega}\atop{q\to p}} P_{B_1}(0,q) |q|^n \left(\dist(q,\partial B_1)\right)^s=\frac{c(n,s)}{2^s},$$ hence
we can rewrite~\eqref{hop4} as
	  \eqlab{ \label{hop2}
			\lim_{{q\in{\mathcal{C}}\Omega}\atop{q\to p}} P_\Omega(0,q) |q|^n \left(\dist(q,\partial \Omega)\right)^s
			=
			\lim_{{q\in{\mathcal{C}}\Omega}\atop{q\to p}} P_{B_1}(0,q) |q|^n \left(\dist(q,\partial B_1)\right)^s.
	}
In view of these observations, we can deduce
Theorem~\ref{DERe99} from the following result:

\begin{theorem}\label{2DERe}
If~$\Omega \subset \Rn$ is an open set with~$C^{1,1}$ boundary,
containing the origin,
with~${\rm{dist}}(0, \partial \Omega)=1$ and $p\in \partial \Omega \cap \partial B_1$,
such that
\begin{equation}\label{uiUiuyOidDFFoi845} 
\lim_{{q\in{\mathcal{C}}\Omega}\atop{q\to p}}
P_\Omega(0,q)\,|q|^n\big(
{\rm{dist}}(q, \partial \Omega)\big)^s
=\frac{c(n,s)}{2^s},\end{equation}
then~$\Omega=B_1$.
\end{theorem}
	
	\smallskip
	
	We point out that, in general, it may happen that
	\eqlab{ \label{OMEGA-OME2-73485}
		\lim_{{q\in{\mathcal{C}}\Omega}\atop{q\to \partial \Omega}} P_\Omega(0,q) |q|^n \left(\dist(q,\partial \Omega)\right)^s = C(n,s,\Omega),
	}
i.e. the limit towards any point of the boundary
may be a constant (depending on $\Omega$), without $\Omega$ being necessarily the unit ball
centered at the origin: for example,
given~$R>1$ and
$x_0\in \partial B_{R-1}$, we consider the domain $\Omega:=
B_R(x_0) $. Notice that $ 0\in B_R(x_0)$ and  that~${\rm dist}(0,\partial B_R(x_0))=R-|x_0|=1$. 
 However, given~$p\in\partial B_R(x_0)$,
taking~$q_t:=p+\frac{t(p-x_0)}{|p-x_0|}$, we have that
\begin{eqnarray*}&&\lim_{t\to0^+}
P_{B_R(x_0)}(0,q_t)\,|q_t|^n\,\big({\rm{dist}}(q_t, \partial B_R(x_0))\big)^s
\\ &&=\lim_{t\to0^+}\frac{c(n,s)\,t^s\,(R^2-|x_0|^2)^s}{(|q_t-x_0|^2-R^2)^s}=\lim_{t\to0^+}\frac{c(n,s)\,t^s\,(2R-1)^s}{((R+t)^2-R^2)^s}\\&&\qquad=\frac{c(n,s)\,(2R-1)^s}{2^s\,R^s},
\end{eqnarray*}
hence this limit is the same for all~$p\in\partial B_R(x_0)$.

It would be interesting to further investigate the geometric implications
of condition~\eqref{OMEGA-OME2-73485}, to
relate boundary limits 
to geometric
properties of the domain
and to classify all the domains 
for which~\eqref{OMEGA-OME2-73485}
holds true. Also, it would be nice to investigate
the specific role played by the Euclidean norm
and understand the case of different norms, including anisotropic
situations and Minkowski norms induced by a convex domain.

\medskip

Concerning the quantity~$\mathfrak c(\Omega)$
in~\eqref{54327yYAS:03eorfijhh6778S} and the setting in \eqref{hop3}, we have the next result.

\begin{lemma}\label{JKMS:0-2rfiugj}
Let~$\Omega$ be a bounded domain with~$C^{1,1}$
boundary such that~${\rm dist}(0,\partial\Omega)=1$
and~$\Omega\subseteq B_R$, and let~$F_\Omega$ be as
in~\eqref{hop3}. Suppose that~\eqref{OMEGA-OME299}
holds true for all functions~$u\in\mathcal H^s(\Omega)$. 

Then
\begin{equation}\label{7768cnsnamcppBB}
\mathfrak c(\Omega)\in\left[\frac{1}{R^{2s}},1\right].
\end{equation}
Moreover, 
\begin{equation}\label{OMABXpal}
		\mathfrak c(\Omega)=1\quad {\mbox{ if and only if}}\quad	
\Omega=B_1.\end{equation}
	\end{lemma}

It is also interesting to observe that
the limit in~\eqref{OMEGA-OME2-73485}
always exists if the domain is~$C^{1,1}$.

\begin{theorem}\label{edcyhBoNo8549yt896549}
Let~$\Omega\subset\R^n$ be an open set with~$C^{1,1}$
boundary, with~$p\in\partial \Omega$ and let~$x_0\in\Omega$.
Then, the limit
\begin{equation*}
\lim_{t\to 0^+} P_\Omega(x_0,p+t\nu(p))\,\big( {\rm dist}(p+t\nu(p),\partial\Omega)\big)^s\end{equation*}
exists and it is finite.
\end{theorem}

The proofs of Theorems \ref{DERe99}, \ref{2DERe}, \ref{edcyhBoNo8549yt896549}  and of Lemma \ref{JKMS:0-2rfiugj}   
 are contained in Section \ref{tre}.

\section{Proofs of the main results}

In this section, we provide the proofs of the
main results of this note, 
 together with other auxiliary results.
\subsection{Proof of Theorem~\ref{THM}}\label{thm11}
\label{uno}
\begin{proof}[Proof of Theorem~\ref{THM}]
\label{TOWMSYY023}
We argue towards a contradiction,
assuming that
$\Omega \setminus B_r \neq \varnothing$.
Then, let~$p\in\Omega \setminus B_r$.
Since~$\Omega$ is open, there exists~$\rho>0$ such that~$B_\rho(p)\subset\Omega$.
Furthermore, one sees that, since~$p\not\in B_r$,
it holds that~$B_\rho(p)\setminus\overline{B_r}\ne\varnothing$.
These observations give that
$$\varnothing\ne B_\rho(p)\setminus\overline{B_r}\subset\Omega\setminus B_r,$$
and therefore, by~\eqref{musta1},
\begin{equation}\label{MISP}
\mu_r(\Omega\setminus B_r)>0.
\end{equation}
Moreover, according to~\eqref{mvp}
and~\eqref{OMEGA}, for any $u\in \mathcal H^s(\Omega)$ with~$u(0)=0$
we have that
	\eqlab{
		\label{eqn1}
		0\,&= \mu_r(\Co \Omega) \,u(0)=
\int_{\Co \Omega} u(y)\, d\mu_r(y) 
			=\int_{\Co B_r} u(y)\, d\mu_r(y)  -\int_{\Omega\setminus B_r} u(y)\, d\mu_r(y)
				\\&= \mu_r(\Co B_r) \,u(0)- \int_{\Omega\setminus B_r} u(y)\, d\mu_r(y)
	=- \int_{\Omega\setminus B_r} u(y)\, d\mu_r(y).}
Now, for every~$x\in\R^n$ we let~$f(x):=|x|^2$, and we define
	\[ R:= \max_{y\in \overline \Omega} |y|.\]  We also consider~$\epsilon>0$ suitably small,
possibly in dependence of~$r$. For concreteness,
we take  
\begin{equation*}
\epsilon:=\frac{r^2}{4}.
\end{equation*}
We exploit Theorem~1.1 in~\cite{DSV14} for this~$\eps$ to obtain the existence of a  function~$f_{r,R} \in C^s_0(\Rn)$ such that
\begin{eqnarray*}&& (-\Delta)^s f_{r,R}=0 \quad{\mbox{ in }}B_{R},\\
{\mbox{and }}&& \| f_{r,R}-f\|_{L^\infty(B_{R})}\le\epsilon=\frac{r^2}4.
\end{eqnarray*}
Then, we define
\eqlab{\label{ustar} u^\star(x):=-f_{r,R}(x)+f_{r,R}(0).}
We remark that, for all~$x\in B_{R}$,
\begin{equation*}
\begin{split}
u^\star(x)=&\;-f_{r,R}(x) + f(x) + f_{r,R}(0) - f(0) -f(x)+f(0)
 \\
\le &\;
-f(x)+f(0) + |f(x)-f_{r,R}(x)|
+|f_{r,R}(0)-f(0)|
\\ 
\le&\; 
-|x|^2+\frac{r^2}2.
\end{split}
\end{equation*}
Hence 
\bgs{ 
\label{Aer} - u^\star(x) \geq |x|^2-\frac{r^2}2 \geq \frac{r^2}2\qquad \quad \mbox{ for all } x\in B_{R}\setminus B_r.}
Since~$\Omega\subset B_R$,
it follows from this and~\eqref{MISP} 
that
\eqlab{ \label{ddd} &
	 \int_{\Omega\setminus B_r}-  u^\star(y)\,d\mu_r(y)
\geq 
\frac{r^2}2\,\mu_r(\Omega\setminus B_r) > 0 .
}
We also point out that~$u^\star(0)=0$ and~$(-\Delta)^s u^\star(x)=
(-\Delta)^s f_{r,R}(x)=0$ for all~$x\in B_{R}$,
 consequently~$u^\star\in{\mathcal{H}}^s(B_{R})\subset
{\mathcal{H}}^s(\Omega)$.

Hence, we can exploit~\eqref{eqn1} with~$u:=u^\star$,
obtaining a contradiction with~\eqref{ddd}.
\end{proof}

Now we present a structurally different proof of Theorem~\ref{THM}
based on fractional potential theory. This proof
is based on the following idea:
if we assume that $\Omega$ has $C^{1,1}$ boundary, using~\eqref{DEF:POISSONKE}
and~\eqref{ballpoi}, we obtain
	\[
	 \int_{\Co \Omega} u(y)\left( \frac{P_{B_r}(0,y)}{\mu_r(\Co \Omega)}- P_\Omega(0,y)\right) dy =0,
	 \] 
	 and this holds for any $u\in C^\infty_0(\Co \Omega)$. These considerations  give that 
		 \eqlab{ \label{intuit}
		 \frac{P_{B_r}(0,y)}{\mu_r(\Co \Omega)}= P_\Omega(0,y)
	  }
	 almost everywhere in $\Co \Omega$. The thesis of Theorem \ref{THM} is therefore equivalent to proving that
\eqref{intuit} holds if and only if $\Omega=B_r.$ 
	 	 The proof of this claim is carried out reasoning by contradiction. Intuitively, looking at \eqref{intuit}, the contradiction is obtained by choosing a point $p^* \in \partial \Omega$, and $p^* \notin \partial B_1$ and taking the limit for $y\in \Co \Omega$ to $p^*$: the left-hand side term will tend to infinity, whereas the right-hand side gives a finite value. 
	 	 
	 	 We note that in the proof we exploit an approximation argument
in order to deal with general domains. The detailed 
exposition follows hereafter.
	 
\begin{proof}[Potential theory proof of Theorem~\ref{THM}]\label{PAKLA1}
We argue by contradiction and suppose that~$\Omega\setminus B_r$
is nontrivial. Then, by sliding a ball inside~$\Omega\setminus B_r$,
we can find a ball~$B^*\subset\Omega\setminus\overline{B_r}$
with~$(\partial B^*)\cap((\partial\Omega)\setminus\overline{ B_r})
\ne\varnothing$, see Figure~\ref{fig:my-label}.
In this way, we can consider a point~$p^*\in
(\partial B^*)\cap((\partial\Omega)\setminus\overline{ B_r})$.
We also take a sequence~$p_j\in {\mathcal{C}}\Omega$ such that~$p_j\to p^*$
as~$j\to+\infty$.
We define~$\varpi:=B_r\cup {B^*}$. 

\begin{figure}[!htb]
        \center{\includegraphics[width=0.7\textwidth]
        {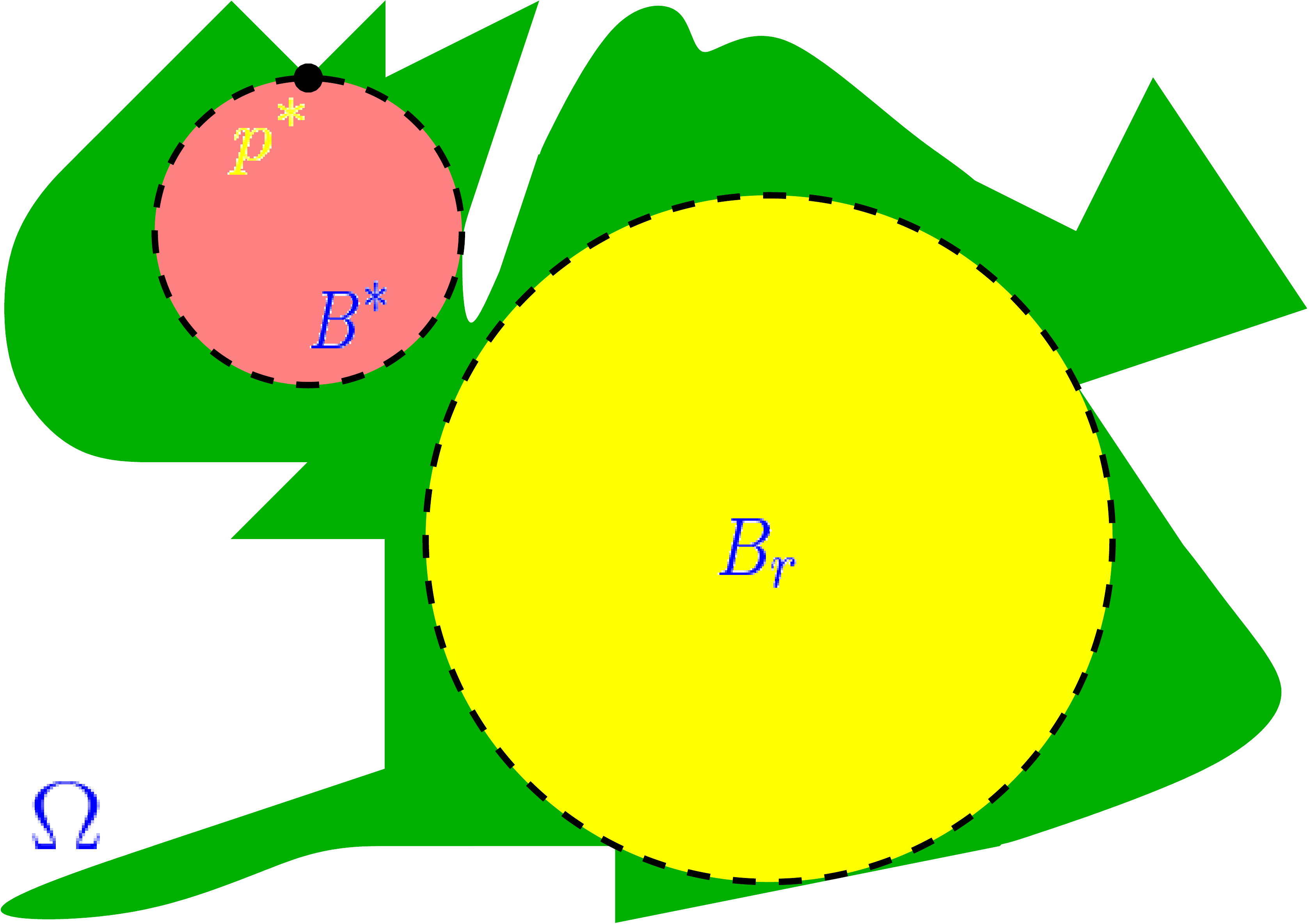}}
        \caption{\label{fig:my-label}\sl What happens if~$\Omega\setminus{B_r}\ne\varnothing$.}
      \end{figure}

The main idea now is to consider, as a test for our contradiction,
a harmonic function in~$\varpi$ formally corresponding to a Dirac mass at~$p_j$.
On the one hand,
this function will reproduce the Poisson kernel~$P_\varpi(\cdot,p_j)$
(as described in Appendix~\ref{POSKE}), which diverges as~$p_j\to p^*$;
on the other hand, the corresponding average in~\eqref{OMEGA}
would converge to a finite
value, thus providing the desired contradiction.

The details of the technical argument go as follows.
We take~$\varphi_{k,p_j}$ as in~\eqref{APPI2}
and thus~$u_{k,p_j}$ as in~\eqref{APPI}.
Given~$j$, we always suppose that~$k$ is large,
possibly in dependence of~$j$, such that~$B_{1/k}(p_j)$
lies well outside~$\Omega$, that is
\begin{equation}\label{4471}
\overline{B_{1/k}(p_j)}\subset {\mathcal{C}}\overline{\Omega}.
\end{equation}
Also, given~$\delta>0$, we take a smooth bounded open set~$\Omega^{(\delta)}$
that contains~$\Omega$ and such that all points of~$\Omega^{(\delta)}$
have distance less than~$\delta$ from~$\Omega$.

In view of~\eqref{4471},
we can take~$\delta$ sufficiently small (possibly in dependence
of~$k$ and~$j$), such that
\begin{equation}\label{44712}
\overline{B_{1/k}(p_j)}\subset {\mathcal{C}}\overline{\Omega^{(\delta)}}.
\end{equation}
We take~$u_{k,p_j,\delta}$ to be the fractional
harmonic function coinciding with~$\varphi_{k,p_j}$
outside~$\Omega^{(\delta)}$.

We claim that
\begin{equation}\label{43}
u_{k,p_j,\delta}\ge u_{k,p_j}.
\end{equation}
For this, we observe that~$u_{k,p_j,\delta}\ge0$, by Maximum Principle.
Hence, since~$u_{k,p_j}=\varphi_{k,p_j}=0$ in $({\mathcal{C}}\varpi)\cap
({\mathcal{C}}B_{1/k}(p_j))$, it follows that~\eqref{43} holds true at least
in~$({\mathcal{C}}\varpi)\cap({\mathcal{C}}\B_{1/k}(p_j))\supseteq
({\mathcal{C}}\varpi)\cap\Omega^{(\delta)}$.

This, and the fact that~\eqref{43} holds true by construction in~$
{\mathcal{C}}\Omega^{(\delta)}$, gives that \eqref{43}
holds in~${\mathcal{C}}\varpi$.
On the other hand, both~$u_{k,p_j,\delta}$
and~$u_{k,p_j}$ are $s$-harmonic in~$\varpi$, hence~\eqref{43}
follows from the Maximum Principle.

Now we claim that
\begin{equation}\label{P089p008}
\lim_{\delta\to0}
\int_{\Co \Omega} u_{k,p_j,\delta}(y) \, d\mu_r (y)=
\int_{\Co \Omega} \varphi_{k,p_j}(y) \, d\mu_r (y)
.\end{equation}
To this end, we observe that the image of~$ \varphi_{k,p_j}$
is~$[0,k]$, and therefore also the image of~$u_{k,p_j,\delta}$
is~$[0,k]$, by Maximum Principle. Then, since
\begin{eqnarray*}
&&\int_{\Co \Omega} u_{k,p_j,\delta}(y) \, d\mu_r (y)\\
&=&\int_{\Co \Omega^{(\delta)}} u_{k,p_j,\delta}(y) \, d\mu_r (y)+
\int_{\Omega^{(\delta)}\setminus \Omega} u_{k,p_j,\delta}(y) \, d\mu_r (y)\\
&=&\int_{\Co \Omega^{(\delta)}} \varphi_{k,p_j}(y) \, d\mu_r (y)+
\int_{\Omega^{(\delta)}\setminus \Omega} u_{k,p_j,\delta}(y) \, d\mu_r (y),
\end{eqnarray*}
one obtains~\eqref{P089p008} by taking the limit.

Therefore, in light of~\eqref{OMEGA}, \eqref{43} and~\eqref{P089p008},
\begin{equation*}
\begin{split}&
u_{k,p_j}(0)
\le
\lim_{\delta\to0}u_{k,p_j,\delta}(0)
=
\lim_{\delta\to0}\frac{1}{\mu_r(\Co \Omega)}
\int_{\Co \Omega} u_{k,p_j,\delta}(y) \, d\mu_r (y)\\&\qquad\qquad=\frac{1}{\mu_r(\Co \Omega)}
\int_{\Co \Omega} \varphi_{k,p_j}(y) \, d\mu_r (y).
\end{split}\end{equation*}
Hence,  
observing that~$0\in B_r\subset\varpi$,
taking limits as~$k\to+\infty$, and recalling
Lemmata~\ref{LE:PDformt} and~\ref{HA:odf45},
\begin{equation}\label{76860092344tuiuu38585}
\begin{split}&P_\varpi(0,p_j)=\lim_{k\to+\infty}
u_{k,p_j}(0)
\le\lim_{k\to+\infty}\frac{1}{\mu_r(\Co\Omega)}
\int_{\Co \Omega} \varphi_{k,p_j}(y) \, d\mu_r (y)\\
&\qquad\qquad
=\frac{	c(n,s)\,r^{2s}}{\mu_r(\Co \Omega)\,(|p_j|^2 -r^2)^s |p_j|^{n} }
.\end{split}\end{equation}
Now, we take limits as~$j\to+\infty$. To this end, we exploit
Theorem~2.13 in~\cite{MR1687746}, that provides a suitable~$c:=c(n,s,\varpi)>0$ such that
$$ P_\varpi(0,p_j)\ge
\frac{c\,\big({\rm{dist}}(0, \partial \varpi)\big)^s}{
\big({\rm{dist}}(p_j, \partial \varpi)\big)^s\,
\big( 1+{\rm{dist}}(p_j, \partial \varpi)\big)^s\,
|p_j|^n}.$$
Hence, since~$p_j\to p^*\in\partial B^*\subseteq\partial\varpi$,
and consequently~${\rm{dist}}(p_j, \partial \varpi)\to0$ as~$j\to+\infty$,
we obtain that
\begin{equation}\label{STAH:001} \lim_{j\to+\infty}P_\varpi(0,p_j)=+\infty.\end{equation}
Plugging this information into~\eqref{76860092344tuiuu38585},
we conclude that
\begin{equation}\label{STAH:002}
+\infty=\lim_{j\to+\infty}
\frac{	c(n,s)\,r^{2s}}{\mu_r(\Co \Omega)\,(|p_j|^2 -r^2)^s |p_j|^{n} }=
\frac{	c(n,s)\,r^{2s}}{\mu_r(\Co \Omega)\,(|p^*|^2 -r^2)^s |p^*|^{n} }.\end{equation}
But
\begin{equation}\label{STAH:003}
\frac{	1}{(|p^*|^2 -r^2)^s |p^*|^{n}}<+\infty,\end{equation}
since~$p^*\in\R^n\setminus\overline{B_r}$.
{F}rom \eqref{STAH:002} and~\eqref{STAH:003}
a contradiction plainly follows, and thus the claim in Theorem~\ref{THM}
is established.
\end{proof}

Interestingly, we point out that the proofs of Theorem~\ref{THM}
presented here hold true in a greater generality and they remain valid, with only notation modifications, in the case in which the measure~$\mu_r(\cdot)$ is replaced by a family of measure~$\mu(\cdot;U)$, possibly varying for every open set~$U$ containing~$B_r$ and which satisfy the following structural properties:
\begin{itemize}
\item the restriction of~$\mu(\cdot;U)$ to~${\mathcal{C}}B_r$
is absolutely continuous with respect to the Lebesgue measure,
\item $\mu(\cdot,B_r)=\mu_r(\cdot)$,
\item $\mu(A;U)\le\mu_r(A)$ for all~$A\subseteq{\mathcal{C}}B_r$.
\end{itemize}
In this framework, the second assumption is mainly needed to guarantee
that the ball satisfies the mean value property with respect to~$\mu(\cdot,B_r)$
(without it,
the result has to allow the possibility that no set satisfies
such a mean value property).

\subsection{Mean value gaps: Proofs of Theorems \ref{STA5}, \ref{7u8j9234905iitt}, \ref{GSTOTRA}} \label{due}

\begin{proof}[Proof of Theorem~\ref{STA5}] 
The fact that~$G_r(B_r)=0$ follows from~\eqref{mvp},
hence we focus on proving that~$G_r(\Omega)\ge1$ otherwise.
The proof is a quantitative
refinement of the potential theoretic argument
presented on page~\pageref{PAKLA1}.
We assume that~$\Omega\setminus\overline{B_r}\ne\varnothing$
and then we put ourselves in the setting of Figure~\ref{fig:my-label}: namely,
we consider a ball~$B^*\subset\Omega\setminus\overline{B_r}$
with~$(\partial B^*)\cap((\partial\Omega)\setminus\overline{ B_r})
\ne\varnothing$, a point~$p^*\in
(\partial B^*)\cap((\partial\Omega)\setminus\overline{ B_r})$,
and a sequence~$p_j\in {\mathcal{C}}\Omega$ such that~$p_j\to p^*$
as~$j\to+\infty$, thus defining~$\varpi:=B_r\cup {B^*}$. 

We take~$\varphi_{k,p_j}$ as in~\eqref{APPI2}
and~$u_{k,p_j}$ as in~\eqref{APPI}.

In this way, by~\eqref{F86243gerre},
\begin{equation}\label{6274-1}
G_r(\Omega)\ge \frac{\displaystyle
\bigg|u_{k,p_j}(0)- \frac{1}{\mu_r(\Co \Omega)} \int_{\Co \Omega} u_{k,p_j}(y)\, d\mu_r (y)\bigg|}{\displaystyle \int_{\Co B_r} |u_{k,p_j}(y)| \, d\mu_r (y)}.\end{equation}
By Lemma~\ref{LE:PDformt}, we know that
\begin{equation}\label{6274-2}
\lim_{k\to+\infty}u_{k,p_j}(0)=P_\varpi(0,p_j).
\end{equation}
Also, by Lemma~\ref{HA:odf45},
\begin{equation}\label{6274-3}
\lim_{k\to+\infty}
\int_{\Co \Omega} 
\frac{	u_{k,p_j}(y)}{(|y|^2 -r^2)^s |y|^{n} }\,dy=
\lim_{k\to+\infty}
\int_{\Co \Omega} 
\frac{	\varphi_{k,p_j}(y)}{(|y|^2 -r^2)^s |y|^{n} }\,dy=
\frac{	1}{(|p_j|^2 -r^2)^s |p_j|^{n} }
.\end{equation}
Furthermore,
since~$\varphi_{k,p_j}\ge0$, it follows from~\eqref{APPI} that~$u_{k,p_j}\ge0$.
This observation and~\eqref{mvp} give that 
\begin{equation*}
u_{k,p_j}(0) = 
\int_{\Co B_r} u_{k,p_j}(y)\, d\mu_r(y)=
\int_{\Co B_r} |u_{k,p_j}(y)| \, d\mu_r(y).
\end{equation*}
As a consequence, by~\eqref{6274-2},
\begin{equation*}
P_\varpi(0,p_j) =\lim_{k\to+\infty}
\int_{\Co B_r} |u_{k,p_j}(y)| \, d\mu_r(y)
.\end{equation*}
Then, we insert this,~\eqref{6274-2} and \eqref{6274-3}
 into~\eqref{6274-1} after taking the limits
as~$k\to+\infty$, and we find that
\begin{equation*}
G_r(\Omega)\ge \frac{
\bigg|P_\varpi(0,p_j)- \displaystyle\frac{c(n,s)r^{2s}}{\mu_r(\Co \Omega)\,(|p_j|^2 -r^2)^s |p_j|^{n} }
\bigg|}{
P_\varpi(0,p_j)}\ge
\frac{
P_\varpi(0,p_j)- \displaystyle\frac{c(n,s) r^{2s}}{\mu_r(\Co \Omega)\,(|p_j|^2 -r^2)^s |p_j|^{n} }
}{
P_\varpi(0,p_j)}
.\end{equation*}
Then, taking the limit as~$j\to+\infty$, it follows from~\eqref{STAH:001}
and~\eqref{STAH:003} that~$G_r(\Omega)\ge1$.
\end{proof}

Now we will use the result in~\cite{DSV14}
to prove Theorem~\ref{7u8j9234905iitt}.

\begin{proof}[Proof of Theorem \ref{7u8j9234905iitt}]
We point out that~$G_r^*(B_r)=0$, as a consequence of~\eqref{mvp},
hence we focus on the proof of~\eqref{5t6t6g6h89g304}.
For this,
the argument that we use here is a suitable modification
and quantification
of the one in the proof of Theorem~\ref{THM}
presented on page~\pageref{TOWMSYY023}, combined with
some rescaling methods.

We denote 
\begin{equation}\label{00557eudhvbb238850} \Omega_r := \frac{\Omega}{r} \qquad{\mbox{and}}\qquad u_r(x):=u(rx).\end{equation}
We take~$R$ so large that~$\Omega\subset B_R$,
and consequently
\begin{equation*}
B_1\subset \Omega_r\subset B_{R/r}.
\end{equation*}
Moreover, in view of~\eqref{SPAZ}, we remark that
\begin{equation*}{\mbox{$u\in\mathcal H^s(\Omega)$
if and only if~$u_r\in\mathcal H^s(\Omega_r)$. }}\end{equation*}
Furthermore, for every~$\Omega'$ which contains~$B_r$,
by~\eqref{musta1}, and using the substitution~$z:=y/r$, we see that
\begin{equation}\label{72yehffPzasg}
\begin{split}
& \int_{\Co \Omega'} u(y)\,d\mu_r(y)\,=\;
c(n,s)\,r^{2s}\,\int_{\Co\Omega'}
\frac{u(y)\,dy}{(|y|^2 -r^2)^s |y|^{n} }\\&\qquad =\;
c(n,s)\,\int_{\Co\Omega'_r}
\frac{u(rz)\,dz}{(|z|^2 -1)^s |z|^{n} } =
\int_{\Co \Omega'_r} u_r(z)\,d\mu_1(z),\end{split}
\end{equation}
where the notation in~\eqref{00557eudhvbb238850} has been used for~$\Omega'$
as well.

In particular, taking~$\Omega':=B_r$ in~\eqref{72yehffPzasg},
\begin{equation}\label{72yehffPzasg1}
\int_{\Co B_r} u(y)\,d\mu_r(y)=
\int_{\Co B_1} u_r(z)\,d\mu_1(z).
\end{equation}
Also, taking~$u:=1$
in~\eqref{72yehffPzasg},
\begin{equation*}
\mu_r(\Co\Omega')=\mu_1(\Co\Omega'_r).
\end{equation*}
This, for~$\Omega':=\Omega$ gives
\begin{equation}\label{72yehffPzasg2}
\mu_r(\Co\Omega)=\mu_1(\Co\Omega_r).
\end{equation}
Making use of this, \eqref{mvp}, \eqref{72yehffPzasg} and~\eqref{72yehffPzasg1}, we get that
\begin{equation*}
\begin{split}&
\frac{\displaystyle u(0)- \frac{1}{\mu_r(\Co \Omega)}
\int_{\Co \Omega} u(y)\, d\mu_r (y)}{
\displaystyle \int_{\Co B_r} u(y) \, d\mu_r (y)}=
\frac{\displaystyle u_r(0)- \frac{1}{\mu_1(\Co \Omega_r)}
\int_{\Co \Omega_r} u_r(y)\, d\mu_1 (y)}{
\displaystyle \int_{\Co B_1} u_r(y)\, d\mu_1 (y)}\\&\qquad
=\frac{\displaystyle \int_{\Co B_1} u_r(y)\, d\mu_1 (y)- \frac{1}{\mu_1(\Co \Omega_r)}
\int_{\Co \Omega_r} u_r(y)\, d\mu_1 (y)}{
\displaystyle \int_{\Co B_1} u_r(y)\, d\mu_1 (y)}
\\&\qquad=1-\frac{\displaystyle
\int_{\Co \Omega_r} u_r(y)\, d\mu_1 (y)}{{\mu_1(\Co \Omega_r)}\,
\displaystyle \int_{\Co B_1} u_r(y)\, d\mu_1 (y)}.
\end{split}\end{equation*}
This gives that
\begin{equation}\label{JSKDLLD:03459596}
G^*_r(\Omega)\ge\sup_{v\in \H^s(\Omega_r)}\left|
1-\frac{\displaystyle
\int_{\Co \Omega_r} v(y)\, d\mu_1 (y)}{{\mu_1(\Co \Omega_r)}\,
\displaystyle \int_{\Co B_1} v(y)\, d\mu_1 (y)}
\right|.
\end{equation}
Now, we take~$f\in C^\infty(\R^n,[0,1])$ with
\sys[f(x)=]{ & 1 && \mbox{ in } B_{1/2}  \cup  \mathcal{C}B_{2R/r}, \\
	&0 && \mbox{ in }B_{R/r}\setminus B_1.}
We let~$\eps>0$ be
sufficiently small and
we exploit Theorem~1.1 in~\cite{DSV14} 
and see that there exists $f_{R/r}\in C^s_0(\Rn)$ such that 
\begin{equation*}
\begin{split}& (-\Delta)^s f_{R/r}=0 \quad{\mbox{ in }}B_{3R/r},\\
{\mbox{and }}\quad& \| f_{R/r}-f\|_{L^\infty (B_{3R/r})}\le
\eps.
\end{split}\end{equation*}
In particular, we have that
$$\| f_{R/r}\|_{L^\infty (\Omega_r\setminus B_1)}\le
\| f_{R/r}\|_{L^\infty (B_{R/r}\setminus B_1)}=
\| f_{R/r}-f\|_{L^\infty (B_{R/r}\setminus B_1)}\le
\eps,$$ whence
\begin{equation}\label{Cabzie6787}
\begin{split}&
\int_{\Co B_1} f_{R/r}(y)\, d\mu_1 (y)
=\int_{\Co \Omega_r} f_{R/r}(y)\, d\mu_1 (y)+
\int_{\Omega_r\setminus B_1} f_{R/r}(y)\, d\mu_1 (y)\\&\qquad\le
\int_{\Co \Omega_r} f_{R/r}(y)\, d\mu_1 (y)+\eps\mu_1(\Omega_r\setminus B_1)\le \int_{\Co \Omega_r} f_{R/r}(y)\, d\mu_1 (y)+\eps
,
\end{split}
\end{equation}
thanks to~\eqref{998y99iuqhjss89}.

Notice also
that~$f_{R/r}\in \mathcal H^s(B_{R/r})\subseteq \mathcal H^s(\Omega_r)$. As a consequence
of this, \eqref{mvp}, \eqref{JSKDLLD:03459596} and~\eqref{Cabzie6787},
we see that
\begin{equation}\label{SM:93}
\begin{split}
G_r^*(\Omega)\,&\ge
\frac{\displaystyle
\int_{\Co \Omega_r} f_{R/r}(y)\, d\mu_1 (y)}{{\mu_1(\Co \Omega_r)}\,
\displaystyle \int_{\Co B_1} f_{R/r}(y)\, d\mu_1 (y)}-1\\&\ge
\frac{1}{{\mu_1(\Co \Omega_r)}}-
\frac{\eps}{{\mu_1(\Co \Omega_r)}\,
\displaystyle \int_{\Co B_1} f_{R/r}(y)\, d\mu_1 (y)}
-1\\&=\frac{1}{{\mu_1(\Co \Omega_r)}}-
\frac{\eps}{{\mu_1(\Co \Omega_r)}\, f_{R/r}(0)}
-1.
\end{split}
\end{equation}
Since
$$ f_{R/r}(0)\ge f(0)-|f_{R/r}(0)-f(0)|\ge f(0)-\eps=1-\eps\ge\frac12,$$
as long as~$\eps$ is small enough,
we deduce from~\eqref{SM:93} that
$$ G_r^*(\Omega)\ge
\frac{1}{{\mu_1(\Co \Omega_r)}}-
\frac{2\eps}{{\mu_1(\Co \Omega_r)}}
-1.$$
Hence, taking~$\eps$ as small as we wish, and recalling~\eqref{72yehffPzasg2},
we obtain that
\begin{equation} \label{7768869978}
G_r^*(\Omega)\ge
\frac{1}{{\mu_r(\Co\Omega)}}-1.
\end{equation}
Now, for any~$\delta>0$ we can consider a smooth domain~$\Omega^{(\delta)}$
such that
$$ \big(B_r\cup\{|x_n|\ge 2\delta\}\big)\cap B_{1/\delta}\subseteq\Omega^{(\delta)}
\subseteq \big(B_r\cup\{|x_n|\ge \delta\}\big)\cap B_{2/\delta}.$$
Then, we have that
\begin{eqnarray*}
\mu_r(\Co\Omega^{(\delta)})&\le&\mu_r\Big(
\big((\mathcal{C}B_r)\cap\{|x_n|< 2\delta\}\big)\cup(\mathcal{C} B_{1/\delta})
\Big)\\&\le&
\mu_r\big((\mathcal{C}B_r)\cap\{|x_n|< 2\delta\}\big)+\mu_r\big(\mathcal{C} B_{1/\delta}\big)\\&\le&
c(n,s)\,r^{2s}\left[
\int_{ (\mathcal{C}B_r)\cap\{|y_n|< 2\delta\} }
\frac{dy}{(|y|^2 -r^2)^s |y|^{n} }+
\int_{ \mathcal{C} B_{1/\delta} }
\frac{dy}{(|y|^2 -r^2)^s |y|^{n} }
\right],
\end{eqnarray*}
which is infinitesimal as~$\delta\to0$.

Using this information and~\eqref{7768869978}, we obtain
$$ \lim_{\delta\to0}G_r^*(\Omega^{(\delta)})\ge
\frac{1}{{\mu_r(\Co\Omega^{(\delta)})}}-1=+\infty,$$
thus establishing~\eqref{5t6t6g6h89g304}.
\end{proof}

Now we consider the fractional
gap defined in~\eqref{Evisd923hhdgur945}
and we show that it reproduces precisely the excess of~$\Omega$
with respect to~$B_r$, measured in terms of~$\mu_r$, thus proving
Theorem~\ref{GSTOTRA}. To do this, we use once more
the result in~\cite{DSV14}.

\begin{proof}[Proof of Theorem~\ref{GSTOTRA}]
Let~$u\in \H^s(\Omega)$ with~$|u|\leq1$ in~$\Omega$. By~\eqref{mvp}
and~\eqref{998y99iuqhjss89},
\begin{equation}\label{HAnrdjerd823}
\begin{split}&
\bigg|\mu_r(\Co \Omega)\,u(0)- \int_{\Co \Omega} u(y)\, d\mu_r (y)\bigg|
\\&\qquad=\; 
\bigg|
\frac{\mu_r(\Co \Omega)}{\mu_r(\Co B_r)} \int_{\Co B_r} u(y)\, d\mu_r (y)
- \int_{\Co \Omega} u(y)\, d\mu_r (y)\bigg|
\\&\qquad\le \; 
\bigg|
\frac{\mu_r(\Co \Omega)}{\mu_r(\Co B_r)} 
- 1\bigg|\,\bigg|
\int_{\Co B_r} u(y)\, d\mu_r (y)\bigg|
+ \int_{\Omega\setminus B_r} |u(y)|\, d\mu_r (y)
\\&\qquad\le\; 
\bigg|
\frac{\mu_r(\Co \Omega)}{\mu_r(\Co B_r)} 
- 1\bigg|\,\mu_r(\Co B_r)\,|u(0)|
+
\mu_r(\Omega\setminus B_r)\\&\qquad=\; 
\big(\mu_r(\Co B_r)-\mu_r(\Co \Omega)\big)
\,|u(0)|
+\mu_r(\Omega\setminus B_r) \\&\qquad\le
2\mu_r(\Omega\setminus B_r).
\end{split}\end{equation}
Now, we take~$R>0$ large enough such that~$\Omega\subset B_R$.
Given~$\delta\in(0,r)$, we take~$f_\delta\in C^\infty(\R^n,\,[0,1/2])$
be such that~$f_\delta=0$ in~$B_{r-\delta}$
and~$f_\delta=1/2$ in~${\mathcal{C}}B_r$.
Fixed~$\eps>0$, we exploit Theorem~1.1 in~\cite{DSV14}: 
in this way, we find a function~$f_{\delta,\eps,R}$ such that~$f_{\delta,\eps,R}\in
\H^s(B_R)\subseteq\H^s(\Omega)$ and
$$ \|f_{\delta,\eps,R}-f_{\delta}\|_{L^\infty(B_R)}\le\eps.$$
Notice that, for every~$x\in\Omega\subset B_R$, we have that
$$ |f_{\delta,\eps,R}(x)|\le |f_{\delta}(x)|+|
f_{\delta,\eps,R}(x)-f_{\delta}(x)|\le\frac12+\eps\leq1,$$
as long as~$\eps$ is sufficiently small,
and thus~$f_{\delta,\eps,R}$ is an admissible function as a competitor
for the supremum in~\eqref{Evisd923hhdgur945}.
Then, by~\eqref{mvp}
and~\eqref{998y99iuqhjss89},
\begin{eqnarray*}{\mathcal{G}}_r(\Omega)
&\ge&\bigg|\mu_r(\Co \Omega)\,f_{\delta,\eps,R}(0)- 
\int_{\Co \Omega} f_{\delta,\eps,R}(y)\, d\mu_r (y)\bigg|\\
&=&\bigg|\mu_r(\Co \Omega)\,f_{\delta,\eps,R}(0)-
\left[\int_{\Co B_r} f_{\delta,\eps,R}(y)\, d\mu_r (y)-
\int_{\Omega\setminus B_r} f_{\delta,\eps,R}(y)\, d\mu_r (y)\right]
\bigg|\\
&=&\bigg|\mu_r(\Co \Omega)\,f_{\delta,\eps,R}(0)-
\left[\mu_r(\Co B_r)\,f_{\delta,\eps,R}(0)-
\int_{\Omega\setminus B_r} f_{\delta,\eps,R}(y)\, d\mu_r (y)\right]
\bigg|\\&=&
\bigg|\big({\mu_r(\Co \Omega)}-\mu_r(\Co B_r)\big)
f_{\delta,\eps,R}(0)
+\int_{\Omega\setminus B_r} f_{\delta,\eps,R}(y)\, d\mu_r (y)\bigg|\\&=&
\bigg|-{\mu_r(\Omega \setminus B_r)}\,
f_{\delta,\eps,R}(0)
+\int_{\Omega\setminus B_r} f_{\delta,\eps,R}(y)\, d\mu_r (y)\bigg|\\&\ge&
\int_{\Omega\setminus B_r} f_{\delta,\eps,R}(y)\, d\mu_r (y)-{\mu_r(\Omega\setminus B_r)}
\,f_{\delta,\eps,R}(0)\\
&\ge&
\int_{\Omega\setminus B_r} \big(f_{\delta}(y)-\eps\big)\, d\mu_r (y)-{\mu_r(\Omega\setminus B_r)}\,
\big(f_{\delta}(0)+\eps\big)\\
&=&
\left(\frac12-\eps\right)\, \mu_r (\Omega\setminus B_r)-\eps{\mu_r(\Omega\setminus B_r)}.
\end{eqnarray*}
For this reason, by sending~$\eps\to0$, we discover that
$$ {\mathcal{G}}_r(\Omega)\ge
\frac{ \mu_r (\Omega\setminus B_r)}{2}.$$
This and~\eqref{HAnrdjerd823} yield the desired result.
\end{proof}

\subsection{Poisson-like measures:
Proofs of Theorems \ref{DERe99}, \ref{2DERe},~\ref{edcyhBoNo8549yt896549}}
\label{tre}
We start this section by providing the proof of the spherical classification result
in Theorem~\ref{2DERe}.

\begin{proof}[Proof of Theorem~\ref{2DERe}]
Given~$q\in\Co \overline \Omega$, we define~$v(x)=P_{\Omega}(x,q)- P_{B_1}(x,q)$
and we know by Lemma~\ref{DENRP} that
\begin{equation*}
\begin{cases}
(-\Delta)^s v(x)=0 & {\mbox{ for every }}x\in B_1,\\
v(x)=0 & {\mbox{ for every }}x\in\Co\overline \Omega,\\
v(x)=P_{\Omega}(x,q)& {\mbox{ for every }}x\in\Omega\setminus B_1
.\end{cases}\end{equation*}
As a consequence,
$$ v(x)=\int_{\Co B_1} v(y)\,P_{B_1}(x,y)\,dy=\int_{\Omega\setminus B_1}
P_{\Omega}(y,q)\,P_{B_1}(x,y)\,dy
$$
thereby
$$ P_{\Omega}(x,q)=P_{B_1}(x,q)
+\int_{\Omega\setminus B_1}
P_{\Omega}(y,q)\,P_{B_1}(x,y)\,dy.$$
For $x=0$, we thus have
	\bgs{ \label{fru1}
		P_\Omega(0,q)-P_{B_1}(0,q) 
		= \int_{\Omega\setminus B_1} P_\Omega(y,q)P_{B_1}(0,y)\, dy.
	}
We consider  a sequence~$q_j\in {\mathcal{C}}\Omega$ such that~$q_j\to p$
as~$j\to+\infty$. Up to taking $j$ large enough,
we have that $\dist(q_j,\partial \Omega)= \dist(q_j,\partial B_1)$.
Accordingly, using \eqref{ballpoi} we have that
	\[
		\lim_{j\to+\infty}
P_{B_1}(0,q_j)\,|q_j|^n\big(
{\rm{dist}}(q_j, \partial B_1)\big)^s =\frac{c(n,s)}{2^s}
	\]
hence, in light of 
 \eqref{uiUiuyOidDFFoi845}, we obtain
	\eqlab{ \label{Conau04co1019-01}
		0= &\lim_{j\to+\infty}
\big(P_\Omega(0,q_j)-P_{B_1}(0,q_j)\big) |q_j|^n\big(
{\rm{dist}}(q_j, \partial \Omega)\big)^s 
\\
	& =  \int_{\Omega\setminus B_1}
P_{\Omega}(y,q_j)\,P_{B_1}(0,y)\,dy
\,|q_j|^n\big(
{\rm{dist}}(q_j, \partial \Omega)\big)^s.
	}
Our goal is now to show that
\begin{equation*}
\Omega\setminus B_1=\varnothing.
\end{equation*}
To this end, we argue by contradiction and we suppose that~$\Omega\setminus B_1\ne\varnothing$.
In particular, we can find~$\bar x\in\Omega\setminus B_1$ and 
 $\rho$ small
such that~$B_\rho(\bar x)\subset\Omega\setminus B_1$
and such that
$$ {\mbox{
$y\in B_{\rho/2}(\bar x)$, 
hence }}
{\rm{dist}}(y, \partial \Omega)\ge\frac\rho2.$$
Moreover, by Lemma~2.13 in~\cite{MR1687746},
for each~$y\in\Omega$,
\begin{eqnarray*}
P_{\Omega}(y,q_j)\ge\frac{c(\Omega)\,
\big( {\rm{dist}}(y, \partial \Omega)\big)^s}{
|y-q_j|^n\,\big(
{\rm{dist}}(q_j, \partial \Omega)\big)^s\,
\big(1+{\bar c(\Omega)}{\rm{dist}}(q_j, \partial \Omega)\big)^s,}\end{eqnarray*}
for  suitable~$c(\Omega)$, ${\bar c(\Omega)}>0$, wherefore
\begin{eqnarray*}
&& \int_{\Omega\setminus B_1}
P_{\Omega}(y,q_j)\,P_{B_1}(0,y)\,dy
\,\big(
{\rm{dist}}(q_j, \partial \Omega)\big)^s\\
&\ge& \int_{\Omega\setminus B_1}
\frac{c(\Omega)\,
\big( {\rm{dist}}(y, \partial \Omega)\big)^s}{
|y-q_j|^n\,
\big(1+ {\bar c(\Omega)} \,{\rm{dist}}(q_j, \partial \Omega)\big)^s}
\,P_{B_1}(0,y)\,dy
\\&\ge&
\int_{B_{\rho/2}(\bar x)}
\frac{c(\Omega)\,
\big( \rho/2\big)^s}{
|y-q_j|^n\,
\big(1+ {\bar c(\Omega)} \,{\rm{dist}}(q_j, \partial \Omega)\big)^s}
\,P_{B_1}(0,y)\,dy.
\end{eqnarray*}
This leads to
\begin{eqnarray*}&& \lim_{j\to+\infty}\int_{\Omega\setminus B_1}
P_{\Omega}(y,q_j)\,P_{B_1}(0,y)\,dy
\,\big(
{\rm{dist}}(q_j, \partial \Omega)\big)^s\\&&\qquad\ge
\int_{B_{\rho/2}(\bar x)}
\frac{c(\Omega)\,
\big( \rho/2\big)^s}{
|y-p|^n}
\,P_{B_1}(0,y)\,dy\\&&\qquad=
\int_{B_{\rho/2}(\bar x)}
\frac{c(\Omega)\,c(n,s)\,
\big( \rho/2\big)^s}{
|y-p|^n\,(|y|^2 -1)^s |y|^{n} }\,dy
>0.\end{eqnarray*}
This is in contradiction with~\eqref{Conau04co1019-01},
and so it proves that $\Omega\setminus B_1=\varnothing$.
Therefore~$\Omega=B_1$,
hence the proof of
Theorem~\ref{2DERe}
is complete.
\end{proof}

Having completed the proof of Theorem~\ref{2DERe},
we use this result to prove Theorem \ref{DERe99},
via the following argument:

\begin{proof}[Proof of Theorem \ref{DERe99}] \label{pagehop5}
We observe, thanks to \eqref{DEF:POISSONKE}, that for any $u\in \H^s(\Omega)$ 
	\bgs{
			u(0)=\int_{\Co\Omega} u(y)\,P_\Omega(0,y)\,dy.
		}
Putting this together with \eqref{OMEGA-OME299}, we get
	\bgs{
		\int_{\Co\Omega} u(y) \left(P_\Omega(0,y)- \frac{F_\Omega(y)}{\mathfrak c(\Omega)} \right) \,dy =0.
		}	 
In particular, we can take~$u$ to be $s$-harmonic in~$\Omega$, with~$u=\phi$ in~$\Co\Omega$, for any~$\phi\in C^\infty_0(\Co \Omega)$, thus obtaining
	\bgs{
		\int_{\Co\Omega} \phi(y) \left(P_\Omega(0,y)- \frac{F_\Omega(y)}{\mathfrak c(\Omega)} \right) \,dy =0.
	}
This implies that 	
	\eqlab{\label{pan1} 
		\frac{F_\Omega(y) }{\mathfrak c(\Omega)} = P_\Omega(0,y) \quad \mbox{ a.e. in } \; \Co \Omega.
		}
Using \eqref{fru3}, we get that
\[\lim_{{q\in{\mathcal{C}}\Omega}\atop{q\to p}} P_\Omega(0,q)|q|^n \dist^s(q,\partial \Omega)= \frac{c(n,s)}{2^s}.\]
This gives that condition~\eqref{uiUiuyOidDFFoi845} 
is satisfied and
then the assumptions of Theorem~\ref{2DERe}
are fulfilled, thus allowing
us to conclude that $\Omega=B_1$.
\end{proof}

We use the  analysis developed in this section to
give the proof of Lemma~\ref{JKMS:0-2rfiugj}.

\begin{proof}[Proof of Lemma~\ref{JKMS:0-2rfiugj}]
By Lemma \ref{DENRP}, we know
that, for all $q\in \Co \overline \Omega$,
	\[ P_{B_1}(0,q) \leq P_\Omega(0,q). \] 
Using this, \eqref{hop3},
\eqref{Nodnsckq2hwyrfy8w} and~\eqref{pan1} (which holds, as a consequence of~\eqref{OMEGA-OME299}),
we see that
\begin{equation}\label{Idbentiti}\begin{split}&
\frac{c(n,s)}{|q|^n\left(\dist(q,\partial B_1)\right)^s
\left(\dist(q,\partial B_1)+2\right)^s }=P_{B_1}(0,q)\le
P_\Omega(0,q)\\&\qquad=\frac{F_\Omega(q)}{
\mathfrak c(\Omega)}=
				\frac{ c(n,s)}{\mathfrak c(\Omega) |q|^n \,\big(
					{\rm{dist}}(q, \partial \Omega)\big)^s\,
					\big(2+{\rm{dist}}(q, \partial \Omega)\big)^s}.
\end{split}\end{equation}
This gives that	\begin{equation}\label{sinsavkrfgp934r}
		\mathfrak c(\Omega) \leq  \frac{\left(\dist(q,\partial B_1)\right)^s \left(\dist(q,\partial B_1)+2\right)^s}{\left(\dist(q,\partial \Omega)\right)^s \left(\dist(q,\partial \Omega)+2\right)^s}  .
	\end{equation}
Now, we take~$p\in(\partial\Omega)\cap(\partial B_1)$
and, for small~$t>0$, we consider~$q_t:=(1+t)p\in\Co\Omega$.
We observe that, for small~$t$, $\dist(q_t,\partial \Omega)=\dist(q_t,\partial B_1)$
and hence,
in view of~\eqref{sinsavkrfgp934r},
we deduce that
\begin{equation*}
		\mathfrak c(\Omega) \leq
\lim_{t\to0^+} \frac{\left(\dist(q_t,\partial B_1)\right)^s \left(\dist(q_t,\partial B_1)+2\right)^s}{\left(\dist(q_t,\partial \Omega)\right)^s \left(\dist(q_t,\partial \Omega)+2\right)^s}  =1.
	\end{equation*}
Using instead that~$\Omega\subseteq B_R$,
and thus~$P_{B_R}(0,q) \geq P_\Omega(0,q)$
for all $q\in \Co \overline B_R$, we obtain that~$\mathfrak c(\Omega)
\geq R^{-2s}$, and we can conclude the desired claim in~\eqref{7768cnsnamcppBB}.

To prove~\eqref{OMABXpal},
we notice that if~$\Omega=B_1$, then the ``$\leq$''
in~\eqref{Idbentiti} reduces to ``$=$''.
Therefore~$\mathfrak c(B_1)=1$, which is one implication of~\eqref{OMABXpal}.

Now, to prove the other implication of~\eqref{OMABXpal},
we assume~$\mathfrak c(\Omega)=1$
and we aim at proving that~$\Omega=B_1$.
With this assumption, we have from~\eqref{Idbentiti}
that
$$ P_\Omega(0,q)=
				\frac{ c(n,s)}{\mathfrak c(\Omega) |q|^n \,\big(
					{\rm{dist}}(q, \partial \Omega)\big)^s\,
					\big(2+{\rm{dist}}(q, \partial \Omega)\big)^s}=
\frac{ c(n,s)}{ |q|^n \,\big(
					{\rm{dist}}(q, \partial \Omega)\big)^s\,
					\big(2+{\rm{dist}}(q, \partial \Omega)\big)^s}
.$$
Hence, taking~$p$ and~$q_t$ as above,
$$ \lim_{t\to0^+}P_\Omega(0,q_t)\,|q_t|^n \,\big(
					{\rm{dist}}(q_t, \partial \Omega)\big)^s
=\lim_{t\to0^+}
\frac{ c(n,s)}{\big(2+{\rm{dist}}(q_t, \partial \Omega)\big)^s}
=\frac{ c(n,s)}{2^s}.$$
This gives that~\eqref{uiUiuyOidDFFoi845} is satisfied,
and therefore, by Theorem~\ref{2DERe},
we obtain that~$\Omega=B_1$, as desired.
\end{proof}

Now we deal with
the proof of
Theorem~\ref{edcyhBoNo8549yt896549}.
For this, we drew inspiration from a method developed
in a different framework
in~\cite{MR1980119} (see in particular Theorem~3.2 there).
As a first step towards the proof of
Theorem~\ref{edcyhBoNo8549yt896549},
we obtain some uniform bounds on the Poisson kernel.

\begin{lemma}\label{0.1lemma}
Let~$\Omega\subset\R^n$ be an open set with~$C^{1,1}$
boundary, with~$p\in\partial \Omega$ and let~$\bar x\in\Omega$.
Let also~$x_{\rm int}$, $x_{\rm ext}\in\R^n$
and~$r_{\rm int}$, $r_{\rm ext}>0$ be such that
\begin{equation}\label{89:8760193d}
B_{r_{\rm int}}(x_{\rm int})\subseteq
\Omega\subseteq \Co B_{r_{\rm ext}}(x_{\rm ext}),\end{equation}
with
\begin{equation}\label{89:8760193d-2}
p\in(\partial B_{r_{\rm int}}(x_{\rm int}))\cap(\partial
B_{r_{\rm ext}}(x_{\rm ext})).\end{equation}
Then,
\begin{equation*}\begin{split}&
\frac{c(n,s)}{2^s\,r_{\rm int}^s\,|\bar x-p|^n}\left(
2r_{\rm int}\nu(p)\cdot(p-\bar x)
-|\bar x-p|^2
\right)^s\le\liminf_{t\to0^+}
P_\Omega(\bar x,p+t\nu(p))\,t^s\\&\qquad
\le\limsup_{t\to0^+}
P_\Omega(\bar x,p+t\nu(p))\,t^s\le
\frac{c(n,s)}{2^s\,r_{\rm ext}^s\,|\bar x-p|^n}\left(
2r_{\rm ext}\nu(p)\cdot(p-\bar x)+
|\bar x-p|^2
\right)^s.\end{split}\end{equation*}\end{lemma}

\begin{proof} By~\eqref{89:8760193d}
and Lemma~\ref{DENRP}, for all~$y\in B_{r_{\rm ext}}(x_{\rm ext})$,
\begin{equation*}
P_{B_{r_{\rm int}}(x_{\rm int})}(\bar x,y)
\le P_\Omega(\bar x,y)\le
P_{\Co B_{r_{\rm ext}}(x_{\rm ext})}(\bar x,y),
\end{equation*}
and therefore, by \eqref{ballpoi}
\begin{equation*}
\frac{c(n,s)}{|\bar x-y|^n}\left(
\frac{r_{\rm int}^2-|\bar x-x_{\rm int}|^2}{|y-x_{\rm int}|^2-r_{\rm int}^2}
\right)^s\le
P_\Omega(\bar x,y)\le
\frac{c(n,s)}{|\bar x-y|^n}\left(
\frac{|\bar x-x_{\rm ext}|^2-r_{\rm ext}^2}{r_{\rm ext}^2-|y-x_{\rm ext}|^2}
\right)^s
.\end{equation*}
In particular, taking~$t>0$ suitably small and~$y:=p+t\nu(p)$,
\begin{equation}\label{67HSA:qowQWS645966}\begin{split}&
\frac{c(n,s)\,t^s}{|\bar x-p-t\nu(p)|^n}\left(
\frac{r_{\rm int}^2-|\bar x-x_{\rm int}|^2}{|p+t\nu(p)-x_{\rm int}|^2-r_{\rm int}^2}
\right)^s\le
P_\Omega(\bar x,p+t\nu(p))\,t^s\\&\qquad
\le
\frac{c(n,s)\,t^s}{|\bar x-p-t\nu(p)|^n}\left(
\frac{|\bar x-x_{\rm ext}|^2-r_{\rm ext}^2}{r_{\rm ext}^2-|p+t\nu(p)-x_{\rm ext}|^2}
\right)^s
.\end{split}\end{equation}
In light of~\eqref{89:8760193d-2},
we point out that~$|p+t\nu(p)-x_{\rm int}|=r_{\rm int}+t$
and~$|p+t\nu(p)-x_{\rm ext}|=r_{\rm ext}-t$. Consequently,
\begin{eqnarray*}
&& |p+t\nu(p)-x_{\rm int}|^2-r_{\rm int}^2=2r_{\rm int}\,t+t^2
\\
{\mbox{and }}&&
r_{\rm ext}^2-|p+t\nu(p)-x_{\rm ext}|^2=2r_{\rm ext}\,t-t^2.
\end{eqnarray*}
This and~\eqref{67HSA:qowQWS645966}
yield that
\begin{equation*}\begin{split}&
\frac{c(n,s)}{|\bar x-p-t\nu(p)|^n}\left(
\frac{r_{\rm int}^2-|\bar x-x_{\rm int}|^2}{
2r_{\rm int}+t}
\right)^s\le
P_\Omega(\bar x,p+t\nu(p))\,t^s\\&\qquad
\le
\frac{c(n,s)}{|\bar x-p-t\nu(p)|^n}\left(
\frac{|\bar x-x_{\rm ext}|^2-r_{\rm ext}^2}{
2r_{\rm ext}-t}
\right)^s
,\end{split}\end{equation*}
from which we obtain that
\begin{equation}\label{878408785859840832123}\begin{split}&
\frac{c(n,s)}{2^s\,r_{\rm int}^s\,|\bar x-p|^n}\left(
r_{\rm int}^2-|\bar x-x_{\rm int}|^2
\right)^s\le\liminf_{t\to0^+}
P_\Omega(\bar x,p+t\nu(p))\,t^s\\&\qquad
\le\limsup_{t\to0^+}
P_\Omega(\bar x,p+t\nu(p))\,t^s\le
\frac{c(n,s)}{2^s\,r_{\rm ext}^s\,|\bar x-p|^n}\left(
|\bar x-x_{\rm ext}|^2-r_{\rm ext}^2
\right)^s.\end{split}\end{equation}
We also remark that
$$ x_{\rm int}=p-r_{\rm int}\nu(p)\qquad{\mbox{and}}\qquad
x_{\rm ext}=p+r_{\rm ext}\nu(p),$$
consequently
\begin{eqnarray*}
&&r_{\rm int}^2-|\bar x-x_{\rm int}|^2=
r_{\rm int}^2-|\bar x-p+r_{\rm int}\nu(p)|^2
=-|\bar x-p|^2-2r_{\rm int}\nu(p)\cdot(\bar x-p)
\\{\mbox{and }}&&
|\bar x-x_{\rm ext}|^2-r_{\rm ext}^2=
|\bar x-p-r_{\rm ext}\nu(p)|^2-r_{\rm ext}^2=
|\bar x-p|^2-2r_{\rm ext}\nu(p)\cdot(\bar x-p).
\end{eqnarray*}
Plugging this information into~\eqref{878408785859840832123}
we obtain
the desired result.
\end{proof}

With this, we can now complete the proof of
Theorem~\ref{edcyhBoNo8549yt896549}.

\begin{proof}[Proof of Theorem~\ref{edcyhBoNo8549yt896549}] Let~$\eta\in(0,1)$ be as small as we wish in what follows.
We consider interior and exterior tangent balls at~$p$, as in~\eqref{89:8760193d}
and~\eqref{89:8760193d-2},
and we define~$\bar x_\eta:= \eta x_{\rm int}+(1-\eta)p$.

Let also
$$ \Psi(t):=
P_\Omega(x_0,p+t\nu(p))\,\big( {\rm dist}(p+t\nu(p),\partial\Omega)\big)^s=
P_\Omega(x_0,p+t\nu(p))\,t^s.$$
We define
$$ \Phi(t,\eta):=P_\Omega(\bar x_\eta,p+t\nu(p))\,t^s$$
and we see that
\begin{equation}\label{9orj88656ogog}
\Psi(t)=\frac{P_\Omega(x_0,p+t\nu(p))}{P_\Omega(\bar x_\eta,p+t\nu(p))}\,\Phi(t,\eta).\end{equation}
We recall the notion of Martin kernel
based at~$x_0$, see e.g.~\cite[Theorems~2 and~3]{MR2365478}
or~\cite[Theorem~4.3]{MR3932105} (see also~\cite{MR1654115}
and the references therein for a comprehensive treatment
of Martin kernels): in this setting, for every~$x\in\Omega$,
$p\in\partial\Omega$,
we can write that
$$ M_\Omega^{x_0}(x,p)=\lim_{y\to p}\frac{P_\Omega(x,y)}{P_\Omega(x_0,y)}.$$
This and~\eqref{9orj88656ogog} give that
\begin{equation*}\label{894885j-2934}\begin{split}&
\limsup_{t\to0^+}\Psi(t)=\frac{1}{M_\Omega^{x_0}(\bar x_\eta,p)}\,
\limsup_{t\to0^+}\Phi(t,\eta)\\ \qquad{\mbox{and}}\qquad&
\liminf_{t\to0^+}\Psi(t)=\frac{1}{M_\Omega^{x_0}(\bar x_\eta,p)}\,
\liminf_{t\to0^+}\Phi(t,\eta).\end{split}\end{equation*}
From this and in light of Lemma~\ref{0.1lemma},
\begin{eqnarray*}
\frac{\displaystyle\limsup_{t\to0^+}\Psi(t)}{\displaystyle\liminf_{t\to0^+}\Psi(t)} &=& \frac{\displaystyle\limsup_{t\to0^+}\Phi(t,\eta)}{\displaystyle\liminf_{t\to0^+}\Phi(t,\eta)}
\\
&\le&\frac{\displaystyle
\frac{c(n,s)}{2^s\,r_{\rm ext}^s\,|\bar x_\eta-p|^n}\left(
2r_{\rm ext}\nu(p)\cdot(p-\bar x_\eta)+
|\bar x_\eta-p|^2
\right)^s}{\displaystyle\frac{c(n,s)}{2^s\,r_{\rm int}^s\,|\bar x_\eta-p|^n}\left(
2r_{\rm int}\nu(p)\cdot(p-\bar x_\eta)
-|\bar x_\eta-p|^2
\right)^s}
\\&=&
\frac{r_{\rm int}^s\,\left(
2\eta r_{\rm ext}\nu(p)\cdot (p-x_{\rm int})
+
|\eta (p-x_{\rm int})|^2
\right)^s}{r_{\rm ext}^s\,\left(
2\eta r_{\rm int}\nu(p)\cdot(p-x_{\rm int})
-|\eta (p-x_{\rm int})|^2
\right)^s}\\&=&
\frac{r_{\rm int}^s\,\left(
2 r_{\rm ext}\nu(p)\cdot (p-x_{\rm int})
+
\eta| (p-x_{\rm int})|^2
\right)^s}{r_{\rm ext}^s\,\left(
2r_{\rm int}\nu(p)\cdot(p-x_{\rm int})
-\eta| (p-x_{\rm int})|^2
\right)^s}.
\end{eqnarray*}
Consequently, by sending~$\eta\to0^+$,
\begin{equation*}
\frac{\displaystyle\limsup_{t\to0^+}\Psi(t)}{\displaystyle\liminf_{t\to0^+}\Psi(t)}
\le
\frac{r_{\rm int}^s\,\left(
2 r_{\rm ext}\nu(p)\cdot (p-x_{\rm int})
\right)^s}{r_{\rm ext}^s\,\left(
2r_{\rm int}\nu(p)\cdot(p-x_{\rm int})
\right)^s}=1,
\end{equation*}
yielding the desired result.
\end{proof}

\begin{appendix}

\section{A note on the fractional mean value formula on balls}\label{UNAryty45666}

In this appendix, we present an auxiliary result
that shows that continuous functions
that are $s$-harmonic in a given domain
satisfy the mean value formula for every ball contained in the domain
(and not only for the balls that are compactly contained in the domain). For this end, we recall definition \eqref{SPAZ}.

\begin{lemma} \label{LemmaJAK:AA344}
Assume that~$u\in \H^s(\Omega)$
and suppose that~$B_r\subset\Omega$. Then~\eqref{mvpx}
holds true.
\end{lemma}

\begin{proof} Let~$\rho\in\left(\frac{r}2,r\right)$. Then~$B_\rho\subset\subset B_r
\subset\Omega$. Therefore, in view of~\eqref{LemmaJAK:AA3440},
we can employ~\eqref{mvpx} with respect to the ball~$B_\rho$,
hence
\begin{equation}\label{OKchxioq123e94-PRE}
u(0)=
c(n,s)\,\int_{\Co B_\rho} \frac{ \rho^{2s}\, u(y)}{(|y|^2-\rho^2)^s|y|^n}\, dy.
\end{equation}
Furthermore, if~$R\ge 2r$ and~$y\in\Co B_R$, we have that
$$ |y|^2-\rho^2=(|y|+\rho)(|y|-\rho)\ge
|y|\left( \frac{|y|}2+\frac{R}2-\rho\right)\ge\frac{|y|^2}{2}.$$
As a result, given any~$\epsilon\in(0,1)$,
taking a suitable~$
R\ge 2r$,
to be chosen sufficiently
large, possibly in dependence of~$\epsilon$, $u$, $r$, $n$ and~$s$,
but independent of~$\rho$, and
exploiting~\eqref{SPAZ}, we see that
$$ 
c(n,s)\,\left|
\int_{\Co B_R} \frac{ \rho^{2s}\, u(y)}{(|y|^2-\rho^2)^s|y|^n}\, dy\right|
\le 2^{s}\;
c(n,s)\,\int_{\Co B_R} \frac{ r^{2s}\, |u(y)|}{|y|^{n+2s}}\, dy\le\epsilon,
$$
and similarly
$$ 
c(n,s)\,\left|
\int_{\Co B_R} \frac{ r^{2s}\, u(y)}{(|y|^2-r^2)^s|y|^n}\, dy\right|
\le \epsilon.$$
This and~\eqref{OKchxioq123e94-PRE} give that
\begin{eqnarray*}&& \left| 
u(0)-
c(n,s)\,\int_{\Co B_r} \frac{ r^{2s}\, u(y)}{(|y|^2-r^2)^s|y|^n}\, dy
\right|\\&=&
\left| 
c(n,s)\,\int_{\Co B_\rho} \frac{ \rho^{2s}\, u(y)}{(|y|^2-\rho^2)^s|y|^n}\, dy-
c(n,s)\,\int_{\Co B_r} \frac{ r^{2s}\, u(y)}{(|y|^2-r^2)^s|y|^n}\, dy
\right|\\&\le&2\epsilon+c(n,s)\,
\left| 
\int_{B_R\setminus B_\rho} \frac{ \rho^{2s}\, u(y)}{(|y|^2-\rho^2)^s|y|^n}\, dy-
\int_{B_R\setminus B_r} \frac{ r^{2s}\, u(y)}{(|y|^2-r^2)^s|y|^n}\, dy
\right|.
\end{eqnarray*}
Hence, after the change of variable~$z:=ry/\rho$ in one integral,
we obtain that
\begin{equation}\label{OKchxioq123e94}
\begin{split} 
& \left| 
u(0)-
c(n,s)\,\int_{\Co B_r} \frac{ r^{2s}\, u(y)}{(|y|^2-r^2)^s|y|^n}\, dy
\right|\\
\le\,&
2\epsilon+c(n,s)\,
\left| 
\int_{B_{Rr/\rho}\setminus B_r} \frac{ r^{2s}\, u(\rho z/r)}{(|z|^2-r^2)^s|z|^n}\, dz
-
\int_{B_R\setminus B_r} \frac{ r^{2s}\, u(y)}{(|y|^2-r^2)^s|y|^n}\, dy
\right|.\end{split}
\end{equation}
Moreover, since~$u$ is continuous, we have that
$$ 
\chi_{ B_{Rr/\rho}\setminus B_r}(z)
\frac{ r^{2s}\, |u(\rho z/r)|}{(|z|^2-r^2)^s|z|^n}\le
\frac{ r^{2s}\, \|u\|_{L^\infty(B_{R})} }{(|z|^2-r^2)^s|z|^n}\in L^1(\R^n).$$
Consequently, by the Dominated Convergence Theorem and the continuity of~$u$,
we can take the limit as~$\rho\nearrow r$ in~\eqref{OKchxioq123e94},
with~$\epsilon$ fixed, concluding that
\begin{equation*}
\begin{split} 
& \left| 
u(0)-
c(n,s)\,\int_{\Co B_r} \frac{ r^{2s}\, u(y)}{(|y|^2-r^2)^s|y|^n}\, dy
\right|\\
\le\,&
2\epsilon+c(n,s)\,
\left| 
\int_{B_{R}\setminus B_r} \frac{ r^{2s}\, u(z)}{(|z|^2-r^2)^s|z|^n}\, dz
-
\int_{B_R\setminus B_r} \frac{ r^{2s}\, u(y)}{(|y|^2-r^2)^s|y|^n}\, dy
\right|\\=\,&2\epsilon.
\end{split}
\end{equation*}
Since~$\epsilon$ can now be taken arbitrarily small, we conclude that
$$ u(0)=
c(n,s)\,\int_{\Co B_r} \frac{ r^{2s}\, u(y)}{(|y|^2-r^2)^s|y|^n}\, dy,$$
as desired.
\end{proof}

\section{Summary of potential theory}\label{POSKE}

We collect here some ancillary results needed for the potential
theoretic proofs of our main results. For comprehensive
treatments of fractional potential theory,
see e.g.~\cite{MR126885,
MR1438304, chengrennest, MR1654115, MR1980119, MR2006232, bogdan1, MR2256481, MR2365478, MR2569321, MR4061422}
and the references therein.

In this appendix, we will always denote by~$\varpi$
a bounded open set with smooth (say, $C^{1,1}$) boundary.
We will also denote by~$P_\varpi$ the fractional Poisson Kernel
of~$\varpi$, see e.g. Theorem 2.1 of~\cite{MR1687746}:
in this way, if~$u$ is $s$-harmonic in~$\varpi$ and~$u=\bar u$
outside~$\varpi$, we have that
\begin{equation}\label{DEF:POISSONKE}
u(x)=\int_{\R^n\setminus\varpi} \bar u(y)\,P_\varpi(x,y)\,dy\qquad
{\mbox{ for all }}x\in\varpi.
\end{equation}
We recall that the fractional Poisson kernel on the ball $B_R(x_0)$ is defined for all $x\in B_R(x_0)$ and $y\in \Co \overline B_R(x_0)$ as
	\eqlab{ \label{ballpoi}
	P_{B_R(x_0)}(x,y)= \frac{c(n,s) (R^2-|x-x_0|^2)^s}{(|y-x_0|^2-R^2)^s)|x-y|^n}.			
	}
We also recall some basic properties of the fractional Poisson Kernel.

\begin{lemma}\label{CONTY}
For each~$x\in\varpi$, the function~$\R^n\setminus\overline\varpi
\ni y\mapsto P_\varpi(x,y)$ is continuous.
\end{lemma}

\begin{proof} Up to a translation, we suppose that~$0\not\in\overline\varpi$, say~$B_\rho\subset\R^n\setminus\varpi$ for some~$\rho>0$,
and we prove continuity at~$0$. For this,
we take a sequence~$y_j$ converging to~$0$ as~$j\to+\infty$
and we call~$G_\varpi$ the Green function of~$\varpi$.
In this way, see e.g. page~231 of~\cite{MR1687746}, we can write (up to a constant depending on $n$ and~$s$, that we omit
for simplicity)
that for $y\in \Co \overline \varpi$, $x\in \varpi$,
\begin{equation}\label{PP} P_\varpi(x,y)= \int_\varpi \frac{G_\varpi(x,z)}{|y-z|^{n+2s}}\,dz.\end{equation}
Then, for every~$z\in\varpi$,
$$ |y_j-z|\ge |z|-|y_j|\ge\rho-|y_j|\ge\frac\rho2,$$
as long as~$j$ is large enough, and accordingly
$$ \frac{|G_\varpi(x,z)|}{|y_j-z|^{n+2s}}\le
\frac{|G_\varpi(x,z)|}{(\rho/2)^{n+2s}}\in L^1(\varpi).$$
This, \eqref{PP} and the Dominated Convergence Theorem yield that
$$\lim_{j\to+\infty} 
P_\varpi(x,y_j)=\lim_{j\to+\infty} 
\int_\varpi \frac{G_\varpi(x,z)}{|y_j-z|^{n+2s}}\,dz=
\int_\varpi \frac{G_\varpi(x,z)}{|z|^{n+2s}}\,dz
=P_\varpi(x,0),$$
as desired.
\end{proof}

Now we take~$p\in\R^n\setminus\overline\varpi$
and~$\varphi\in C^\infty_0(B_1,[0,1])$,
with~$\varphi$ even and
$$\int_{\R^n}\varphi(x)\,dx=1.$$
For each~$k\in\N$, we define
\begin{equation}\label{APPI2} \varphi_{k,p}(x):=k^n\varphi\left( k(x-p)\right)=
k^n\varphi\left( k(p-x)\right).\end{equation}
We also take~$u_{k,p}$ such that
	\syslab[]{\label{98900101389hd}
	&(-\Delta)^s u_{k,p}= 0 &&\mbox{ in } \varpi,\\
&u_{k,p}
=\varphi_{k,p} && \mbox{ in }\R^n\setminus\varpi. 
}
In view of~\eqref{DEF:POISSONKE},
for every~$x\in\varpi$,
\begin{equation}\label{APPI}
u_{k,p}(x)=\int_{\R^n\setminus\varpi} \varphi_{k,p}(y)\,P_\varpi(x,y)\,dy.
\end{equation}
In the next result, we prove that~$u_{k,p}$ provides a pointwise approximation
of the Poisson Kernel.

\begin{lemma}\label{LE:PDformt}
For every~$x\in\varpi$,
\begin{equation*}\label{7832765490123}\lim_{k\to+\infty}u_{k,p}(x)
=P_\varpi(x,p).\end{equation*}
\end{lemma}

\begin{proof} Up to a translation, we can suppose that~$p=0$.
Then, using the symbol ``$*$'' to denote
the convolution for the given~$x\in\varpi$, 
from Lemma~\ref{CONTY}, \eqref{APPI} and the theory of
approximation of the identity (see e.g.
Theorem~9.9
in~\cite{zygmund}),
$$ \lim_{k\to+\infty}u_{k,0}(x)=
\lim_{k\to+\infty}
\int_{\R^n\setminus\varpi} \varphi_{k,0}(y)\,P_\varpi(x,y)\,dy
=\lim_{k\to+\infty} \varphi_{k,0}*P_\varpi(x,\cdot)=P_\varpi(x,0),$$
as desired.
\end{proof}

The approximation of the identity method inside an averaged
formula produces instead the following result:

\begin{lemma}\label{HA:odf45}
Assume that~$|p|>r$. Then,
$$\lim_{k\to+\infty}
\int_{\Co \Omega} 
\frac{	\varphi_{k,p}(y)}{(|y|^2 -r^2)^s |y|^{n} }\,dy=
\frac{	1}{(|p|^2 -r^2)^s |p|^{n} }
.$$
\end{lemma}

\begin{proof} We define
$$ \Psi(x):=
\frac{	\chi_{{\mathcal{C}}\Omega}(x)}{(|x|^2 -r^2)^s |x|^{n} }$$
and we remark that~$\Psi$ is continuous at~$p$
(since~$|p|>r$).
Consequently, by the theory of
approximation of the identity (see e.g.
Theorem~9.9
in~\cite{zygmund}),
\begin{eqnarray*}&&
\lim_{k\to+\infty}
\int_{\Co \Omega} 
\frac{	\varphi_{k,p}(y)}{(|y|^2 -r^2)^s |y|^{n} }\,dy=
\lim_{k\to+\infty}
\int_{\R^n} 
k^n\varphi\left( k(p-y)\right)\,\Psi(y)\,dy\\&&\qquad=
\lim_{k\to+\infty}
\int_{\R^n} 
k^n\varphi(kz)\,\Psi(p-z)\,dy=\lim_{k\to+\infty}(\varphi_{k,0}*\Psi)(p)=\Psi(p),
\end{eqnarray*}
which yields the desired result.
\end{proof}

The next result recalls the monotonicity properties
of the Poisson kernel with respect to set inclusion.

\begin{lemma}\label{DENRP}
Let~$\varpi_1\subseteq\varpi_2$.
Then,
\begin{equation}\label{HAJS:1-09-00-01}
P_{\varpi_1}(x,y)\le P_{\varpi_2}(x,y)\qquad{\mbox{ for all $x\in\varpi_1$ and~$y\in\Co \overline\varpi_2$.}}
\end{equation}
Moreover, given~$y\in\Co \overline\varpi_2$,
the function~$\R^n\ni x\mapsto v(x):=P_{\varpi_2}(x,y)- P_{\varpi_1}(x,y)$ satisfies
\begin{equation}\label{HAJS:1-09-00-02}
\begin{cases}
(-\Delta)^s v(x)=0 & {\mbox{ for every }}x\in\varpi_1,\\
v(x)=0 & {\mbox{ for every }}x\in\Co \overline\varpi_2,\\
v(x)=P_{\varpi_2}(x,y)& {\mbox{ for every }}x\in\varpi_2\setminus\varpi_1
.\end{cases}\end{equation}
\end{lemma}

\begin{proof} Let~$y\in\left(\Co \overline\varpi_2\right)\subseteq\left(\Co \overline\varpi_1\right)$.
Since, for all~$i\in\{1,2\}$,
$$ \begin{cases}
(-\Delta)^s P_{\varpi_i}(x,y)=0 &{\mbox{ if }}x\in{\varpi_i},\\
P_{\varpi_i}(x,y)=\delta_y(x)&{\mbox{ if }}x\in\Co{\varpi_i},
\end{cases}$$
the claim in~\eqref{HAJS:1-09-00-02} follows by subtraction.

{F}rom~\eqref{HAJS:1-09-00-02}, it also follows that~$v\ge0$ in~$\Co\varpi_1$,
and thus, for every~$x\in\varpi_1$,
$$ v(x)=\int_{\Co\varpi_1} v(y)\,P_{\varpi_1}(x,y)\,dy\ge0,$$
that gives~\eqref{HAJS:1-09-00-01}.
\end{proof}

\end{appendix}

\end{document}